\documentclass[10pt]{amsart}

\usepackage{amsmath}
\usepackage{mathtools}      % improves amsmath
\usepackage{thmtools}
\usepackage{mathabx} % makes many symbols (as $\leq$) more beautiful; must be loaded after mathtools 
\usepackage{enumitem}
\usepackage{soul} % undeline, strikeout etc
\usepackage[nosort]{cite}
\usepackage{graphicx}
\usepackage[colorlinks=true,backref=page,bookmarksopen=false,bookmarksdepth=3]{hyperref} % Links in the pdf. Backref option: each bibliographical entry denotes where it was cited
\renewcommand*{\backref}[1]{}
\renewcommand*{\backrefalt}[4]{\quad \tiny 
  \ifcase #1 (\textbf{\color{red}NOT CITED.})%
  \or    (Cited on page~#2.)%
  \else   (Cited on pages~#2.)%
  \fi}
\usepackage{xcolor}
\hypersetup{
    colorlinks,
    linkcolor={red!50!black},
    citecolor={blue!50!black},
    urlcolor={blue!80!black}
}

\newcommand{\doi}[1]{\href{http://dx.doi.org/#1}{DOI}} %{DOI:{#1}}}
  
\numberwithin{equation}{section}

\declaretheorem[numberwithin=section]{theorem} 
\declaretheorem[sibling=theorem]{proposition} 
\declaretheorem[sibling=theorem]{lemma}
\declaretheorem[sibling=theorem]{corollary}

\declaretheorem[sibling=theorem, style=definition]{remark}

\declaretheorem[sibling=theorem, style=definition]{example}

\def\R{\mathbb R}

\def\Z{\mathbb Z}

\def\<{\langle}
\def\>{\rangle}

\def\tilde{\widetilde}

\begin{document}

\title[Regularity of cocycles and Baire category]{Complete regularity of linear cocycles and the Baire category of the set of Lyapunov-Perron regular points}

\author{Jairo Bochi}
\address{Department of Mathematics \\ The Pennsylvania State University \\ University Park, PA 16802, USA}
\email{bochi@psu.edu}

\author{Yakov Pesin}
\address{Department of Mathematics \\ The Pennsylvania State University \\ University Park, PA 16802, USA}
\email{pesin@math.psu.edu}

\author{Omri Sarig}
\address{Department of Mathematics \\ Faculty of Mathematics and Computer Sciences\\ Weizmann Institute of Science \\ 234 Herzl Street, Rehovot 7610001 Israel}
\email{omri.sarig@weizmann.ac.il}


\begin{abstract}
Given a continuous linear cocycle $\mathcal{A}$ over a homeomorphism $f$ of a compact metric space $X$, we investigate its set $\mathcal{R}$ of Lyapunov-Perron regular points, that is, the collection of  trajectories of $f$ that obey the conclusions of the Multiplicative Ergodic Theorem. We obtain results roughly saying that the set $\mathcal{R}$ is of first Baire category (i.e., meager) in $X$, unless some rigid structure is present. In some settings, this rigid structure forces the Lyapunov exponents to be defined everywhere and to be independent of the point; that is what we call complete regularity. 
\end{abstract}

\date{September 3, 2024. This revision: January 20, 2026.}

\maketitle

\vspace*{-5mm} % otherwise the TOC won't fit in the 1st page

\setcounter{tocdepth}{2}
\tableofcontents

\date{\today}

%%%%%%%%%%%%%%%%%%%%%%%%%%%%%%%%%%%%%%%%%%%%%
\section{Introduction} \label{s.intro}
%%%%%%%%%%%%%%%%%%%%%%%%%%%%%%%%%%%%%%%%%%%%%

Linear cocycles consist of products of matrices driven by some dynamical system. A precise description of the asymptotics of these products is available for those orbits
that are Lyapunov-Perron (or LP) regular. The multiplicative ergodic theorem of Oseledets says that the set $\mathcal{R}$ of LP-regular points is always large in a measure-theoretical sense: it has full measure with respect to any invariant probability measure. However, in many situations, the set $\mathcal{R}$ is of first Baire category (i.e., meager) and so is small in a topological sense. Manifestations of this phenomenon in specific settings can be found in the works \cite{Tian,FerreiraV,Varandas}. As for derivative cocycles, \cite[Theorem~3.14]{ABC} asserts that $\mathcal{R}$ is a meager set for $C^1$-generic diffeomorphisms. 

Our objective is to investigate the Baire category of the set $\mathcal{R}$ of Lyapunov-Perron regular points. More specifically, we aim to identify sufficient conditions under which $\mathcal{R}$ is meager. We want these conditions to be as broadly applicable as possible.

As a guiding principle, if the set $\mathcal{R}$ is non-meager, then some underlying rigid structure should account for this fact. Here is a result that corroborates this principle (Corollary~\ref{cor-criterion}): If the dynamics has the property that stable (or unstable) sets of periodic orbits are dense, then for every continuous cocycle, the set $\mathcal{R}$ of LP-regular points is meager, unless all periodic orbits have exactly the same Lyapunov spectrum. The hypothesis on the dynamics in Corollary~\ref{cor-criterion} is quite general (allowing for nonuniformly hyperbolic maps, for instance), and there are no regularity assumptions on the cocycle other than continuity. 

Under the strong assumption of minimality of the base dynamics, we obtain an even sharper dichotomy (Theorem~\ref{t.dichotomy_minimal}): for every continuous cocycle, the set $\mathcal{R}$ of LP-regular points is either meager or it is the whole space. Furthermore, in the latter case, the cocycle will have extra properties that we call \emph{complete regularity}.

Given a continuous linear cocycle $\mathcal{A}$ over any topological dynamical system $f$, complete regularity is defined by two requirements. They are: (1) every trajectory of $f$ is LP-regular, (2) the Lyapunov spectrum of $\mathcal{A}$ is independent of the point~$x$. A restriction of a cocycle $\mathcal{A}$ to a periodic orbit of $f$ is a simple example of a completely regular cocycle. 

If a cocycle is completely regular, then the Oseledets splitting is defined everywhere. We prove (Corollary~\ref{c.DS}) that the Oseledets splitting is uniformly dominated and, in particular, continuous. 

It becomes apparent that complete regularity is pertinent to our study of the Baire category of the set $\mathcal{R}$ of LP-regular points. Our most general criterion for meagerness of $\mathcal{R}$ (Theorem~\ref{thm-irreg-criterion}) is actually formulated in terms of complete regularity. If the base dynamics $f$ is hyperbolic and the cocycle $\mathcal{A}$ is H\"older continuous, then we have a sharp dichotomy again (Theorem~\ref{t.dichotomy}): either $\mathcal{R}$ is meager or the cocycle is completely regular.

Complete regularity can be characterized in terms of the Sacker-Sell spectrum of the cocycle: see \cite{JohnZamp} and the discussion in Section~\ref{ss.SS} below. 

Thus, motivated to understand the notion of complete regularity, we investigated whether the two requirements in the definition of complete regularity are independent. The short answer is yes: Section~\ref{s.independence} exhibits examples of cocycles where one of the two conditions is satisfied but not the other. These examples may be of interest in other contexts.

%%%%%%%%%%%%%%%%%%%%%%%%%%%%%%%%%%%%%%%%%%%%%
\section{Preliminaries} \label{e.prelim}
%%%%%%%%%%%%%%%%%%%%%%%%%%%%%%%%%%%%%%%%%%%%%

In this section, we recall some basic notions about linear cocycles and Lyapunov exponents. 
% Details and proofs can be found in textbooks such as \cite{Ar,BP,Viana}.
Despite the fact that much of the material of this section makes sense in a measurable setting, we will assume continuity and compactness throughout.

\subsection{Linear cocycles over continuous transformations} 

Let $X$ be a topological space and $f: X\to X$ a homeomorphism. A function $\mathcal{A}=\mathcal{A}(x,n)$, $x\in X$ and $n\in\mathbb{Z}$ with values in $\mathrm{GL}(d,\mathbb{R})$ is called a \emph{continuous linear cocycle} if the following properties hold:
\begin{enumerate}
\item for each $x\in X$, we have $\mathcal{A}(x,0)=\mathrm{Id}$, and given $n,k\in\mathbb{Z}$,
\[
\mathcal{A}(x, n+k)=\mathcal{A}(f^k(x), n)\mathcal{A}(x, k);
\]
\item for every $n\in\mathbb{Z}$, the function $\mathcal{A}(\cdot,n)\colon X\to \mathrm{GL}(d,\mathbb{R})$ is continuous.
\end{enumerate}

Given a continuous function $A\colon X\to \mathrm{GL}(d,\mathbb{R})$, we define a cocycle by
\[
\mathcal{A}(x, n) \coloneqq 
\begin{cases}
A(f^{n-1}(x))\cdots A(f(x))A(x) &\text{if $n>0$}\\ \text{Id} &\text{if $n=0$},\\
{A(f^n(x))}^{-1}\cdots {A(f^{-2}(x))}^{-1}{A(f^{-1}(x))}^{-1}
&\text{if $n<0$}.
\end{cases}
\]
We call the map $A(x)$ the \emph{generator} of the cocycle $\mathcal{A}(x,n)$ and we say that the cocycle is \emph{generated} by the function $A(x)$. Every cocycle $\mathcal{A}(x,n)$ is generated by the function $A(x) \coloneqq \mathcal{A}(x,1)$.

\subsection{Lyapunov exponents and regularity of cocycles}

Given a continuous cocycle $\mathcal{A}$ over the homeomorphism $f$, we say that a point $x \in X$ is \emph{regular} if
there exist a list of numbers 
\begin{equation}\label{e.LE}
\chi_1(x)>\chi_2(x)>\cdots>\chi_{s(x)}(x)
\end{equation}
and a splitting
\begin{equation}\label{e.Oseledets}
\mathbb{R}^d=\bigoplus_{i=1}^{s(x)}E_i(x)
\end{equation}
such that the following conditions are satisfied:
\begin{itemize}
\item for every $i$ with $1 \le i \le s(x)$, the two-sided sequence 
\[
\frac1n \log\|\mathcal{A}(x,n)v\|
\]
converges to $\chi_i(x)$ as $n \to \pm \infty$ uniformly in $v \in S^{d-1} \cap E_i(x)$;
\item for every partition of the set $\{1,\dots,s(x)\}$ into two nonempty (disjoint) subsets $I$, $J$, we have
\[
\lim_{n \to \pm \infty} \frac{1}{|n|} \log \measuredangle \left( 
\mathcal{A}(x,n) \left(\bigoplus_{i \in I} E_i(x)\right), 
\mathcal{A}(x,n) \left(\bigoplus_{j \in J} E_j(x)\right) \right) = 0 \, ,
\]
where the angle $\measuredangle(E,F)$ between two subspaces is defined as the least nonnegative angle between two vectors in the corresponding spaces.
\end{itemize}

If $x$ is a regular point, it can be shown that the \emph{Lyapunov exponents} \eqref{e.LE} and the \emph{Oseledets splitting} \eqref{e.Oseledets} are uniquely defined. Let $\mathcal{R}$ denote the set of regular points. It can be shown that $\mathcal{R}$ is a Borel set and that the number $s(x)$, the Lyapunov exponents $\chi_i(x)$, and Oseledets spaces $E_i(x)$ all depend Borel measurably on $x$, and that they satisfy the invariance relations:
\[
s(f^n(x)) = s(x) \, , \quad 
\chi_i(f^n(x)) = \chi_i(x) \, , \quad 
E_i(f^n(x)) = \mathcal{A}(x,n) E_i(x) \, .
\]
According to the Multiplicative Ergodic Theorem of Oseledets, $\mu(\mathcal{R})=1$ for every $f$-invariant Borel probability measure $\mu$ on $X$. If $\mu$ is ergodic, then the Lyapunov exponent $\chi_i(x)$ is the same for $\mu$-almost every $x$, and this common value is denoted $\chi_i(\mu)$.

The notion of regularity as defined above is equivalent to \emph{Lyapunov-Perron regularity} (or \emph{LP regularity}). The latter requires that the cocycle $\mathcal A$ must be forward and backward regular. We refer the reader to \cite{BP} for these definitions. We note that our requirements \eqref{e.LE} and \eqref{e.Oseledets} imply that every regular point is both forward and backwards regular; we stress that simultaneous forward and backward regularity does not imply regularity.

Here are some other properties that we will need: 
if $x$ is forward regular, then 
\begin{align}
\label{e.oseledets_norm}
\lim_{n\to+\infty}\frac1n\log\|\mathcal{A}(x,n)\| &= \chi_1(x) \, , \quad \text{and} \\
\label{e.oseledets_conorm}
\lim_{n\to+\infty}\frac1n\log m(\mathcal{A}(x,n))&=\chi_{s(x)}(x) \, , 
\end{align}
where 
\[
m(L) \coloneqq \|L^{-1}\|^{-1}
\] 
is the \emph{co-norm} of the linear map $L$.
Furthermore, if $x \in \mathcal{R}$,
\begin{align}
\label{e.oseledets_norm_subbundle}
\lim_{n\to\pm\infty}\frac1n\log\|\mathcal{A}(x,n)|E_i(x)\| &= \chi_i(x) \, , \quad \text{and} \\
\label{e.oseledets_conorm_subbundle}
\lim_{n\to\pm\infty}\frac1n\log m(\mathcal{A}(x,n)|E_i(x))&=\chi_i(x) \, .
\end{align}

Given a regular point $x \in \mathcal{R}$, the dimension $d_i(x)$ of the subbundle $E_i(x)$ is called the \emph{multiplicity} of the Lyapunov exponent $\chi_i(x)$.
The collection 
\[
\mathrm{Sp}(x) \coloneqq \big\{(\chi_i(x),d_i(x)) \  :  \  1 \le i \le s(x)\big\}
\]
is called the \emph{Lyapunov spectrum} of the cocycle $\mathcal{A}$ at $x$. 
Sometimes we will also use the \emph{list} of values of the Lyapunov exponent
\[\tilde\chi_1(x)\ge\cdots\ge\tilde\chi_d(x)\] in which distinct values $\chi_i(x)$ are repeated according to their multiplicities $d_i(x)$, $1\le i\le s(x)$.
When we want to stress the dependence on the cocycle, we use more explicit notations such as
$\mathrm{Sp}(\mathcal{A},x)$,  $\chi_i(\mathcal{A},x)$, etc.

\subsection{Cohomology}

Let $\mathcal{A}$ and $\mathcal{B}$ be two cocycles over the same homeomorphism $f$, and of the same dimension. We say that they are \emph{continuously cohomologous} if there exists a continuous conjugacy between them, that is, a continuous map $C \colon X \to \mathrm{GL}(d,\R)$ such that 
\[
\mathcal{A}(x,n) = C(f^n x)^{-1} \mathcal{B}(x,n) C(x) \, .
\]
Since $C$ is continuous, it is tempered, and hence continuously cohomologous cocycles have the same set of Lyapunov--Perron regular points, and the same Lyapunov spectra. 

\subsection{The exterior power of cocycles} \label{ss.exterior}

Consider the vector space of all alternating $i$-forms on $\R^d$, that is, the space of all multilinear maps from $\R^d \times \cdots \times \R^d$ (with $i$ copies) to $\R$ which are alternating. 
The dual of this vector space is denoted $(\mathbb{R}^d)^{\wedge i}$. 
There exists a canonical alternating multilinear map 
\[
(v_1, \dots, v_i) \in \R^d \times \cdots \times \R^d \mapsto v_1\wedge\cdots\wedge v_i \in (\mathbb{R}^d)^{\wedge i} \, , 
\]
called the \emph{exterior product}.
Given a linear operator $L$ of 
$\mathbb{R}^d$, we define the \emph{$i$-fold exterior power}  $L^{\wedge i}$ as the unique linear operator of $(\mathbb{R}^d)^{\wedge i}$ such that
\[
L^{\wedge i} (v_1\wedge\cdots\wedge v_i)=(Lv_1\wedge\cdots\wedge Lv_i) \, .
\]

Let $\mathcal{A}$ be a cocycle over a homeomorphism $f$ of a topological space $X$. For each $i=1,\dots,n$ consider the cocycle 
$\mathcal{A}^{\wedge i} \colon X\times\Z\to \mathrm{GL}(n,\mathbb{R})^{\wedge i}$ defined by
$\mathcal{A}^{\wedge i}(x,n) = \mathcal{A}(x,n)^{\wedge i}$.
We call $\mathcal{A}^{\wedge i}$ the \emph{$i$-fold exterior power cocycle} of $\mathcal{A}$.

One can show (see \cite[p.~211]{Ar}, \cite[Section 3.1]{BP}, or \cite{Rag}) that 
\begin{enumerate}
\item if $x\in X$ is a LP-regular point for the cocycle $\mathcal{A}$, then it is also a LP-regular point for the exterior power cocycle $\mathcal{A}^{\wedge i}$;
\item the list of values of the Lyapunov exponents of the cocycle $\mathcal{A}^{\wedge i}$ is 
\begin{equation}\label{e.ex-power-lp}
\{\tilde\chi_k(\mathcal{A}^{\wedge i},x)\} = \big\{\tilde\chi_{j_1}(\mathcal{A},x)+\cdots+\tilde\chi_{j_i}(\mathcal{A},x) \ : \ j_1>\cdots >j_i \big\};
\end{equation} 
in particular,
\begin{equation}\label{e.power-top}
\tilde\chi_1(\mathcal{A}^{\wedge i},x) = \tilde\chi_{1}(\mathcal{A},x)+\cdots+\tilde\chi_{i}(\mathcal{A},x) \, .
\end{equation}
\end{enumerate}
%where the sum is taken over all $k$-tuples $\{i_1,\dots,i_k\}$ satisfying $i_1>\cdots >i_k$.

\subsection{Uniform splittings}\label{ss.dominations}

Let $f$ be a homeomorphism of a compact metric space $X$.
Let $\mathcal{A}$ be a continuous linear cocycle over $f$.
Suppose that for each $x \in X$, we are given a splitting 
\begin{equation}\label{e.DS}
\R^d=\bigoplus_{i=1}^s E_i(x) 
\end{equation}
where each $\dim E_i(x)$ is independent of the point.
Furthermore, assume that this splitting is $\mathcal{A}$-invariant in the sense that $\mathcal{A}(x,n) E_i(x) = E_i(f^n(x))$.
We say that splitting is \emph{dominated} if there exists $n_0>0$ such that for every point $x \in X$ and each $i$ with $0 < i < s$, we have:
\begin{equation}\label{e.def_domination}
m \big( \mathcal{A}(x,n_0)|E_i(x) \big) > 2 \big \|  \mathcal{A}(x,n_0)|E_{i+1}(x) \big\|  
%\text{ for each }  i \in \{1,\dots, s-1\} 
\, .
\end{equation}
Equivalently, for all choices of unit vectors $v_i\in E_i(x)$, we have:
\[
\|\mathcal{A}(x,n_0)v_i\| > 2 \|\mathcal{A}(x,n_0)v_{i+1}\|  
%\text{ for each }  i \in \{1,\dots, s-1\} 
\, .
\]
We say that the subbundle $E_i$ \emph{dominates} the subbundle $E_{i+1}$.

Let us emphasize that domination is a \emph{uniform} property: the amount of time $n_0$ needed for the validity of the inequalities above is the same for every point $x$. 
In fact, this assumption is strong enough to force continuity (see \cite[Section B.1]{BDV} or \cite[Section 2.1]{CroPo}):

\begin{proposition}\label{p.DS_is_cont}
If \eqref{e.DS} is a dominated splitting, then each subbundle $E_i(x)$ depends continuously on $x \in X$.
\end{proposition}

In particular, the angle between different subbundles of a dominated splitting is bounded below. Actually, for every partition of the set $\{1,\dots,s\}$ into two nonempty (disjoint) subsets $I$, $J$, we have
\[
\inf_{x \in X} \measuredangle \left( 
\bigoplus_{i \in I} E_i(x) , 
\bigoplus_{j \in J} E_j(x) \right) > 0 \, .
\]

\smallskip

Let us introduce two other uniform notions related to domination. 
Again, let $\mathcal{A}$ be a continuous linear cocycle over $f$, and suppose we are given an $\mathcal{A}$-invariant splitting \eqref{e.DS} into subspaces of constant dimensions. We say that the splitting is \emph{absolutely dominated} if there exists $n_0>0$ such that for all pairs of points $x,y \in X$ and each $i$ with $0 < i < s$, 
\[
m \big( \mathcal{A}(x,n_0)|E_i(x) \big) > 2 \big \|  \mathcal{A}(y,n_0)|E_{i+1}(y) \big\|  
% \text{ for each }  i \in \{1,\dots, s-1\} 
\, .
\]
In particular, the splitting is dominated. 

Finally, an $\mathcal{A}$-invariant splitting $\R^d = E^\mathrm{u}(x) \oplus E^\mathrm{s}(x)$ into two subbundles of constant dimensions is called \emph{uniformly hyperbolic} if there exists $n_0>0$ such that for all point $x \in X$, 
\[
m \big (\mathcal{A}(x,n_0)|E^\mathrm{u}(x) \big) > 2 > \frac{1}{2} > \big \|  \mathcal{A}(x,n_0)|E^\mathrm{s}(x) \big\|  \, .
\]
Note that a uniformly hyperbolic splitting is always absolutely dominated and, in particular, continuous. 

We will need the following characterization of uniform hyperbolicity:

\begin{proposition}\label{p.Mane}
Let $f$ be a homeomorphism of a compact metric space~$X$.
Let $\mathcal{A}$ be a continuous linear cocycle over $f$.	
For each $x \in X$, define subspaces
\begin{align*}
E^\mathrm{u}(x) &\coloneq \left\{v \in \R^d \  : \sup_{n \ge 0} \|\mathcal{A}(x,n) v\| < \infty \right\} \, , \\
E^\mathrm{s}(x) &\coloneq \left\{v \in \R^d \  : \sup_{n \ge 0} \|\mathcal{A}(x,-n) v\| < \infty \right\} \, .
\end{align*}
If these two subbundles have constant dimensions and form a splitting $\R^d = E^\mathrm{u}(x) \oplus E^\mathrm{s}(x)$ at every point $x \in X$, then this splitting is uniformly hyperbolic.
\end{proposition}

The statement above corresponds to Ma\~n\'e~\cite[Corollary~1.2]{Man} specialized to bundles of constant dimensions. 
We note that related results were obtained by Sacker and Sell \cite{SS74} and Selgrade \cite{Sel};
see \cite[Chapters~1 and 6]{ChiLat}, \cite[Chapter~6]{Wen}, and \cite{BlumLat} for more information.

\subsection{Vector bundle setting}\label{ss.bundle}

If $\mathcal{A}$ is a $d$-dimensional linear cocycle over $f \colon X \to X$ with generator $A$, we can consider the following transformation $T_\mathcal{A}$ 
of the product $X \times \R^d$:
\[
T_\mathcal{A}(x , v) = (f(x), A(x) v) \, .
\]
This is a skew-product over $f$ which is linear on fibers. 
The tranformation $T_\mathcal{A}$ completely determines the cocycle.

This motivates considering the more general situation where the product space $X \times \R^d$ is replaced with a vector bundle over $X$ of fiber dimension $d$, and the skew-product $T_\mathcal{A}$ is replaced with a continuous vector bundle automorphism $F$ that factors over $f$. 
The concepts discussed above can be extended to this setting in a straightforward manner. 

The fiber bundle setting allows us, for example, to discuss \emph{derivative cocycles}: in this case, $f$ is a $C^1$ diffeomorphism and we consider the induced automorphism $F=Df$ on the tangent bundle. 

For simplicity of presentation, in this paper we focus on linear cocycles (on trivial vector bundles $X \times \R^d$), but actually the statements remain correct for general vector bundles, unless noted otherwise.

\subsection{Subadditivity}

Recall that a sequence of functions $\varphi_n \colon X \to \R$ is called \emph{subadditive} with respect to some dynamics $f \colon X \to X$ if, for all $n,m \ge 1$,
\[
\varphi_{n+m} \le \varphi_n \circ f^m + \varphi_m \, . 
\]
A standard example is the sequence $\varphi_n(x) = \log \|\mathcal{A}(x,n)\|$,
where $\mathcal{A}$ is cocycle over~$f$.

\begin{proposition}[{Semi-uniform subadditive ergodic theorem \cite[Theorem~1]{Sch}, \cite[Theorem~1.7]{Stu}}]\label{p.SUSAET}
Let $f$ be a continuous transformation of a compact metric space~$X$.
Consider a subadditive sequence of continuous functions $\varphi_n \colon X \to \R$. Then:
\begin{equation}\label{e.erg_max}
\lim_{n \to \infty} \sup_{X} \frac{\varphi_n}{n} = 
\sup_{\mu \in \mathcal{M}_f} \lim_{n \to \infty} \int_X \frac{\varphi_n}{n} \, d\mu \, ,
\end{equation}
where $\mathcal{M}_f$ denotes the set of $f$-invariant Borel probability measures. Furthermore, the supremum in the right hand side is attained at some ergodic measure.
\end{proposition}

We note that variations of this result have been found by several authors (see e.g.\ \cite[\S~3]{Furman}). The most complete version is \cite[Theorem~A.3]{Mor}, which shows that upper-semicontinuity is sufficient, and provides other characterizations of the quantity \eqref{e.erg_max}.

\begin{corollary}\label{c.maybe_Oxtoby}
Let $f$ be a continuous transformation of a compact metric space $X$.
If $\varphi$ is a continuous function such that $\int \varphi \, d\mu = c$ for every $\mu \in \mathcal{M}_f$, where $c$ is a constant, then the Birkhoff averages of $\varphi$ with respect to $f$ converge uniformly to~$c$.
\end{corollary}

Corollary~\ref{c.maybe_Oxtoby} follows directly from Proposition~\ref{p.SUSAET}, or from the usual proof that unique ergodicity implies uniform convergence in the ergodic theorem. 

%%%%%%%%%%%%%%%%%%%%%%%%%%%%%%%%%%%%%%%%%%%%%
\section{Completely regular cocycles} \label{s.CR}
%%%%%%%%%%%%%%%%%%%%%%%%%%%%%%%%%%%%%%%%%%%%%

\subsection{Definition and basic properties}\label{ss.CR_def}

Let $X$ be a compact metric space and $f:X\to X$ a homeomorphism. Let $\mathcal{A}\colon X\times\mathbb{Z}\to \mathrm{GL}(d,\mathbb{R})$ be a continuous cocycle over $f$. We say that $\mathcal{A}$ is \emph{completely regular} if it satisfies the following two conditions:
\begin{enumerate}
\item\label{i.CR1} every point $x\in X$ is LP-regular;
\item\label{i.CR2} the Lyapunov exponents are independent of the point $x\in X$; more precisely, there are numbers $\chi_1>\cdots>\chi_s$ and positive integers $d_1,\dots,d_s$ with $\sum_{i=1}^s d_i=d$ such that $\mathrm{Sp}(\mathcal{A},x)=\{(\chi_i,d_i), i=1,\dots,s\}$ for every $x\in X$. 
\end{enumerate}

Therefore, if a cocycle is completely regular, then at every $x$ we have an Oseledets splitting 
\begin{equation}\label{e.Oseledets_const}
\R^d=\bigoplus_{i=1}^s E_i(x) 
\end{equation}
where the number of subbundles and their dimensions are independent of the point. 
In general, the Oseledets splitting is only measurable and there is no a priori information on the speed of convergence in \eqref{e.oseledets_norm_subbundle} and \eqref{e.oseledets_conorm_subbundle}. However, under the assumption of complete regularity, we obtain better properties, as the following result shows.

\begin{theorem}\label{t.uniformity}
Let $\mathcal{A}$ be a completely regular cocycle over a homeomorphism $f$ of a compact metric space $X$. 
Let \eqref{e.Oseledets_const} be its Oseledets splitting.
Then for each $i$ with $1 \le i \le s$ the following statements hold:
\begin{enumerate}
\item\label{i.continuity} the subspace $E_i(x)$ depends continuously on $x$;
\item\label{i.uniformity} the limits in \eqref{e.oseledets_norm_subbundle} and \eqref{e.oseledets_conorm_subbundle} are both uniform with respect to $x$; more precisely, for all $\varepsilon>0$ we can find $N$ such that
\begin{equation}\label{e.com-reg-char}
\begin{aligned}
\chi_i-\varepsilon&\le\frac1n\log m(\mathcal{A}(x,n)|E_i(x)) & \\
&\le\frac1n\log\|\mathcal{A}(x,n)|E_i(x)\|&\le \chi_i+\varepsilon \, .
\end{aligned}
\end{equation}
for every $|n|>N$ and $x\in X$.
% where $m(L)=\|L^{-1}\|^{-1}$ denotes the conorm of a linear map $L$.
\end{enumerate}
\end{theorem}

Before proving the statement above in full generality, we establish first the case $s=1$, where actually a weaker hypothesis suffices:

\begin{proposition}\label{p.one_exponent}
Let $\mathcal{A}$ be a continuous linear cocycle over a homeomorphism $f$ of a compact metric space $X$. Suppose that there exists a number $\chi \in \R$ such that for every $f$-invariant Borel probability measure $\mu$ on $X$, for $\mu$-almost every $x\in X$, all Lyapunov exponents at $x$ are equal to $\chi$. 
Then, for all $x \in X$,
\[
\lim_{n\to \pm \infty} \frac1n\log m(\mathcal{A}(x,n)) = \lim_{n\to \pm \infty} \frac1n\log\|\mathcal{A}(x,n)\|
\]
and these limits are uniform with respect to $x$.
\end{proposition}

\begin{proof}
It is sufficient to consider positive times $n$, since negative times can be dealt with using the inverse cocycle. 
Consider the subadditive sequence of continuous functions 
$\varphi_n(x)=\log\|\mathcal{A}(x,n)\|$.
By assumption, the sequence 
$\frac{\varphi_n(x)}{n}$ converges pointwise to the constant $\chi$. It follows from Proposition~\ref{p.SUSAET} that
\[
\lim_{n\to\infty}\sup_{x\in X}\frac{\varphi_n(x)}{n}=\chi.
\]
Similarly, the sequence 
$\psi_n(x)=\log m(\mathcal{A}(x, n))$ is superadditive and $\frac{\psi_n(x)}{n}$ converges pointwise to the constant $\chi$, implying that
\[
\lim_{n\to\infty}\inf_{x\in X}\frac{\psi_n(x)}{n}=\chi.
\]
Since $\psi_n\le\varphi_n$, we conclude that both limits $\frac{\varphi_n}{n}\to\chi$ and 
$\frac{\psi_n}{n}\to\chi$ are uniform. 
\end{proof}

\begin{proof}[Proof of Theorem~\ref{t.uniformity}]
Let $\mathcal{A}$ be a completely regular cocycle, with Oseledets splitting $\R^d=\bigoplus_{i=1}^s E_i(x)$. Let $\chi_1>\cdots>\chi_s$ be the Lyapunov expoents, which are independent of $x \in X$.
The case $s=1$ is covered by Proposition~\ref{p.one_exponent}, so assume that $s \ge 2$.

Let us prove statement~\eqref{i.continuity}, i.e.,  continuity of the subbundles $E_i(x)$. 
Fix numbers $c_1,\dots,c_{s-1}$ such that $\chi_i>c_i>\chi_{i+1}$. For each $i\in\{1,\dots,s-1\}$, consider the cocycle $\tilde{\mathcal{A}}_i=e^{-c_i}\mathcal{A}$. This cocycle admits an invariant measurable splitting $E^\mathrm{u}_i(x)\oplus E^\mathrm{s}_i(x)$, where
$$
E^\mathrm{u}_i(x) = E_1(x)\oplus\cdots\oplus E_i(x), \quad 
E^\mathrm{s}_i(x) = E_{i+1}(x)\oplus\cdots\oplus E_s(x) \, .
$$
Note that all Lyapunov exponents along $E^\mathrm{u}_i(x)$ (respectively, $E^\mathrm{s}_i(x)$) are positive (respectively, negative).
An application of Proposition~\ref{p.Mane} shows that the splitting  $E^\mathrm{u}_i \oplus E^\mathrm{s}_i$ is uniformly hyperbolic.  In particular, the bundles $E_i^\mathrm{u}$ and $E_i^\mathrm{s}$ are continuous. It follows that $E_1 = E^\mathrm{u}_1$ and $E_s = E^\mathrm{s}_s$ are continuous. Furthermore, for each~$i$, the bundle $E_i$ is the transverse intersection of $E_i^\mathrm{u}$ and $E_{i-1}^\mathrm{s}$, and therefore, it is continuous as well. This completes the proof of statement~\eqref{i.continuity}. 

Now we apply Proposition~\ref{p.one_exponent} (or actually its extension to the vector bundle setting) to the restrictions $\mathcal{A}(x, n)|E_i(x)$. We then obtain the desired uniformity property, statement~\eqref{i.uniformity}.
\end{proof}

As an important consequence of  Theorem~\ref{t.uniformity}, we obtain the property of absolute domination (defined in Section~\ref{ss.dominations} above):

\begin{corollary}\label{c.DS}
Given a completely regular cocycle, its Oseledets splitting \eqref{e.Oseledets_const} is either trivial (i.e., $s=1$) or absolutely dominated.
\end{corollary}

\begin{proof}
Let $\mathcal{A}$ be a completely regular cocycle with spectrum $(\chi_i,d_i)_{i=1,\dots,s}$. Assume that $s > 1$. Fix a positive number $\varepsilon$  small enough so that $\chi_i-\varepsilon > \chi_{i+1} + \varepsilon$ for each $i$ with $0< i < s$. Let $N$ be as in Statement~\eqref{i.uniformity} in Theorem~\ref{t.uniformity}. Then, for all $n>N$, all $x,y \in X$, and  all $i$ with $0< i < s$, we have 
\[
m \big( \mathcal{A}(x,n)|E_i(x) \big) \ge e^{(\chi_i - \varepsilon) n } > e^{(\chi_{i+1} + \varepsilon) n }  \ge \big\| \mathcal{A}(y,n)|E_{i+1}(y) \big\| \, .
\]
This proves absolute domination.
\end{proof}

It is clear that complete regularity is invariant under continuous cohomology.
Here is another basic property:

\begin{proposition}\label{p.ep}
The exterior power of a completely regular cocycle is completely regular.
\end{proposition}

\begin{proof}
This follows immediately from the discussion in Subsection~\ref{ss.exterior}.
\end{proof}

\subsection{Examples}\label{ss.examples}

We present some examples of completely regular cocycles. 

\begin{example}
A one-dimensional continuous linear cocycle with generator $A$ (which in this case takes values in $\R\setminus\{0\}$) is completely regular if and only if the forward and backward limits of Birkhoff averages of the function $\varphi(x) \coloneqq \log |A(x)|$ exist everywhere and are equal to some constant~$c$. Equivalently (see Corollary~\ref{c.maybe_Oxtoby}), $\int \varphi \, d \mu = c$ for every $f$-invariant Borel probability measure~$\mu$. 
% In the terminology of \cite{BJ}, we say that $\varphi-c$ is a \emph{weak coboundary}.
Another equivalent condition is that $\varphi-c$ is a uniform limit of continuous coboundaries (see \cite{MOP} or \cite[Lemma~3]{BJ}). 
\end{example}

\begin{example}\label{ex1}
If the space $X$ is finite and $f \colon X \to X$ is a cyclic permutation, that is, if the dynamics consists of a single periodic trajectory, then any cocycle over $f$ is completely regular. 
\end{example}

\begin{example}\label{ex3}
Now assume that $f$ is a uniquely ergodic homeomorphism, and that the cocycle $\mathcal{A}$ takes values in $\mathrm{SL}^{\pm}(2,\R)$ (the subgroup of $\mathrm{GL}(2,\R)$ formed by matrices with determinant $\pm 1$). Let $\mu$ be the unique invariant probability measure for $f$, which is necessarily ergodic, and 
let $\chi_1 = \chi_1(\mu)$ and $\chi_2 = \chi_2(\mu)$.
By the hypothesis on the determinant, 
$\chi_1 \ge 0 \ge \chi_2 = -\chi_1$.

There are three mutually exclusive cases:
\begin{enumerate}
\item\label{i:zero} $\chi_1 = 0 = \chi_2$.
\item\label{i:UH} $\chi_1 > 0 > \chi_2$ and the cocycle is uniformly hyperbolic.
\item\label{i:NUH} $\chi_1 > 0 > \chi_2$ and the cocycle is not uniformly hyperbolic.
\end{enumerate}

\begin{proposition}\label{p.trichotomy}
Consider a continuous $\mathrm{SL}^{\pm}(2,\R)$-cocycle over a uniquely % strictly 
ergodic homeomorphism. Then, the cocycle is completely regular in cases \eqref{i:zero} and \eqref{i:UH} above, and it is not completely regular in case~\eqref{i:NUH}.
\end{proposition}

\begin{proof}
Note that $m(\mathcal{A}(x,n)) = \|\mathcal{A}(x,n)\|$, since determinants are $\pm 1$. 
Consider case \eqref{i:zero}. 
By Proposition~\ref{p.one_exponent},
$\frac{1}{n}\log\|\mathcal{A}(x,n)\|$ converges uniformly to zero. It follows that every $x\in X$ is LP-regular with zero Lyapunov exponents. In particular, the cocycle is completely regular.

Next, consider the uniformly hyperbolic case \eqref{i:UH}. The hyperbolic bundles are continuous, invariant, and one-dimensional. Unique ergodicity of the base dynamics then implies that the cocycle is completely regular.

Finally, consider the nonuniformly hyperbolic case \eqref{i:NUH}. If the cocycle were completely regular, then, by Corollary~\ref{c.DS}, it would admit a dominated splitting into one-dimensional bundles. Since $|\det A(x)| \equiv 1$, that splitting would actually be uniformly hyperbolic, contrary to our assumption. This shows that case~\eqref{i:NUH} forbids complete regularity.
\end{proof}

If $f$ is an irrational rotation of the circle, simple examples of cocycles in situation \eqref{i:NUH} were obtained by Herman \cite{Herman}.
The existence of cocycles of type \eqref{i:NUH} for \emph{any} uniquely ergodic dynamics with a non-atomic invariant measure was established only recently by Avila and Damanik \cite{AD}, answering a question of Walters \cite{Walters}. 

We point out that Proposition~\ref{p.trichotomy} is essentially contained in \cite{Furman}. For related results, see \cite{DL}.
\end{example}

\begin{remark}
Given a cocycle $\mathcal{A}$ over a homeomorphism $f$, and an $f$-invariant and ergodic Borel probability measure $\mu$, we define the Lyapunov spectrum $\mathrm{Sp}(\mathcal{A},\mu)$ as equal to $\mathrm{Sp}(\mathcal{A},x)$ for $\mu$-almost every $x \in X$.
Proposition~\ref{p.one_exponent} shows that if $\mathrm{Sp}(\mathcal{A},\mu)$ is the same for all ergodic measures $\mu$ \emph{and consists of a single Lyapunov exponent}, then the cocycle $\mathcal{A}$ is completely regular. On the other hand, the condition in italics cannot be dropped: consider for instance any cocycle of type~\eqref{i:NUH} in Proposition~\ref{p.trichotomy}.
\end{remark}

\begin{example}\label{ex.amenable_matrix}
Let $f$ be a homeomorphism of a compact metric space $X$, $c_1 > \cdots > c_s$ a list of distinct real numbers, and $d_1, \dots, d_s$ positive integers such that $d =d_1 + \dots + d_s$.
Suppose that $B : X \to \mathrm{GL}(d, \R)$ is a continuous block-diagonal matrix function 
\begin{equation}\label{e.normal_form}
B(x) = 
\begin{pmatrix} 
	B_1(x)	&			&	0	\\  
			& \ddots	&		\\ 
	0		&			&	B_s(x)
\end{pmatrix},
\end{equation}
where each diagonal block $B_i(x)$ has dimension $d_i \times d_i$ and has block-triangular form
\begin{equation}\label{e.normal_form_subblocks} 
B_i(x) 
=
e^{c_i} \, 
\begin{pmatrix} 
	U_{i,1}(x)	&	*			& 			& *						\\ 
				& U_{i,2}(x)	& \ddots 	& 						\\
				&				& \ddots	& *						\\
	0			&				&			& U_{i,\ell_i}(x)
\end{pmatrix},
\end{equation}
where each matrix $U_{i,j}(x)$ is orthogonal. 

Consider the cocycle $\mathcal{B}(x,n)$ with generator $B(x)$. Then, this cocycle is completely regular, with Lyapunov exponents $c_1, \dots, c_s$, and respective multiplicities $d_1, \dots, d_s$. 
\end{example}

\begin{example}\label{ex.amenable_bundle}
We extend the previous example. 
Let $f$ be a homeomorphism of a compact metric space $X$, let $\mathcal{E}$ be a continuous vector bundle over $X$ whose fibers $\mathcal{E}(x)$ have dimension $d$, and let $F \colon \mathcal{E} \to \mathcal{E}$ be an automorphism that factors over $f$. We say that $F$ admits a \emph{continuous amenable reduction}\footnote{The terminology comes from \cite{KalSad}, but the two definitions are not exactly the same as we assume extra properties.}
if there exist
\begin{itemize}
\item real numbers $c_1 > \cdots > c_s$;
\item positive integers $d_1, \dots, d_s$ such that $d =d_1 + \dots + d_s$;
\item a continuous family on inner products on the fibers $\mathcal{E}(x)$ (called a ``metric'');
\item a continuous $F$-invariant splitting
\[
\mathcal{E}_1(x) \oplus \cdots \oplus \mathcal{E}_s(x)  =  \mathcal{E}(x) 
\]
where each subbundle $\mathcal{E}_i$ has dimension $d_i$;
\item for each $i \in \{1,\dots, s\}$, a continuous $F$-invariant field of flags 
\[
\{0\} = \mathcal{E}_{i,0}(x) \subset \mathcal{E}_{i,1}(x) \subset \cdots \subset \mathcal{E}_{i,\ell_i}(x)= \mathcal{E}_i(x)
\]
\end{itemize}
such that, for each $i \in \{1,\dots, s\}$ and $j \in \{1,\dots,\ell_i\}$, the action induced by $F$ on the quotient bundle $\mathcal{E}_{i,j}/\mathcal{E}_{i,j-1}$ expands the (induced) metric by a constant factor $e^{c_i}$. 

In this case, the vector bundle automorphism $F$ is completely regular, with Lyapunov exponents $c_1, \dots, c_s$, and respective multiplicities $d_1, \dots, d_s$. 

Existence of a continuous amenable reduction is invariant with respect to continuous cohomology.
The cocycle $\mathcal{B}$ in Example~\ref{ex.amenable_matrix} admits a continuous amenable reduction where the spaces $\mathcal{E}_{i,j}$ are independent of $x$. In particular, if a cocycle $\mathcal{A}$ is continuously cohomologous to a cocycle $\mathcal{B}$ as in Example~\ref{ex.amenable_matrix}, then $\mathcal{A}$ admits a continuous amenable reduction. The converse is not true, since there is no guarantee that the subbundles $\mathcal{E}_{i,j}$ are trivializable. 
\end{example}

\begin{example}\label{ex.tangency}
Consider any surface diffeomorphism of class $C^1$ admitting a homoclinic tangency, that is, a hyperbolic fixed point $p$ whose stable and unstable manifolds are tangent at a point $q$. Let $X$ be the the closure of the orbit of $q$ and consider the derivative cocycle $Df$ restricted to $X$. This cocycle has a unique LP-regular point, namely $p$. 

This example has the property that all points in $X$ are both forward and backwards regular and the Lyapunov exponents for future and past agree, but nevertheless the cocycle is not completely regular. \end{example}

\subsection{H\"older continuous cocycles over hyperbolic dynamics} \label{ss.Livsic}

We will discuss a setting where it is possible to obtain a concrete description of all completely regular cocycles. The result is stated in terms of vector bundle automorphisms (see Section~\ref{ss.bundle}) and continuous amenable reduction (see Example~\ref{ex.amenable_bundle}). Actually, we will work with H\"older continuous vector bundles and their automorphisms (see \cite[\S~2.2]{KalSad}, \cite[\S~2.2--2.3]{BochiGaribaldi}).

\begin{theorem}\label{t.amenable_reduction}
Assume that $X$ is a compact smooth manifold and $f:X\to X$ is a transitive $C^{1+\delta}$ Anosov diffeomorphism. 
Let $\mathcal{E}$ be a H\"older vector bundle over $X$ and let $A \colon \mathcal{E} \to \mathcal{E}$ be a H\"older automorphism that factors over $f$. Then the following conditions are equivalent:
\begin{enumerate}
\item\label{i.tar_1} 
$F$ is completely regular;
\item\label{i.tar_2} 
$F$ admits a continuous amenable reduction;
\item\label{i.tar_3} 
$\mathrm{Sp}(F,p) = \mathrm{Sp}(F,q)$ for all pairs of periodic points $p,q$.
\end{enumerate}
Furthermore, if any of these conditions hold, then the subbundles and the metric composing the amenable reduction of $F$ can be taken H\"older continuous, with the same H\"older exponent as the initial data.
\end{theorem}

\begin{proof}
The implications \eqref{i.tar_2} $\Rightarrow$ \eqref{i.tar_1} $\Rightarrow$ \eqref{i.tar_3} are automatic (and independent of $f$ being Anosov or $F$ being H\"older). We are left to prove that \eqref{i.tar_3} $\Rightarrow$ \eqref{i.tar_2}.

So assume \eqref{i.tar_3}, that is, all periodic points have a common Lyapunov spectrum $(\chi_i,d_i)_{i=1}^s$. As shown by Guysinsky \cite{Guy} and DeWitt--Gogolev \cite[Theorems~1.2 and 3.10]{DG}, 
the cocycle admits a dominated splitting $\mathcal{E}_1 \oplus \cdots \oplus \mathcal{E}_s$ where $\dim \mathcal{E}_i = d_i$.
In our setting, the subbundles $\mathcal{E}_i$ of a dominated splitting are H\"older continuous (see \cite{Brin}). Consider the restriction of the automorphism $F$ to one of these subbundles $\mathcal{E}_i$. 
By a result of Kalinin and Sadovskaya \cite[Theorem 3.9]{KalSad} (see also \cite[\S~13.7]{Sad24}), there exists a H\"older continuous metric (i.e.\ family of inner products) on $\mathcal{E}_i(x)$ and an invariant H\"older continuous field of flags 
\[
\{0\} = \mathcal{E}_{i,0}(x) \subset \mathcal{E}_{i,1}(x) \subset \cdots \subset \mathcal{E}_{i,\ell_i}(x)= \mathcal{E}_i(x)
\]
such that the action induced by $F$ on the quotient bundle $\mathcal{E}_{i,j}/\mathcal{E}_{i,j-1}$ expands the (induced) metric by a constant factor $e^{\varphi_i(x)}$, where $\varphi_i$ is a H\"older function on $X$. In our situation, $\varphi_i$ has the same integral with respect to all invariant probability measures supported on periodic orbits, namely $\chi_i$. Therefore we can  apply Livsic theorem and write $\varphi_i = \chi_i + \psi_i \circ f - \psi_i$, where $\psi_i$ is a H\"older continuous function. We rescale the Riemannian metric inside $\mathcal{E}_i(x)$ by multiplying it by $e^{-\psi_i(x)}$, so that the new expansion factor is the constant $e^{\chi_i}$. Then we combine these metrics by declaring the subbundles $\mathcal{E}_i(x)$ to be orthogonal. This shows that $F$ admits a H\"older continuous ameanable reduction.
\end{proof}

Let us note that in dimension $d=1$, Theorem \ref{t.amenable_reduction} is equivalent to Livsic theorem for the group $(\R,+)$. Since Livsic theorem does not hold for functions that are merely continuous instead of H\"older (see \cite[p.~215]{BJ} or \cite{Kocsard}), the regularity hypothesis on the automorphism is necessary for the validity of our statement. 

On the other hand, we believe that Theorem \ref{t.amenable_reduction} holds for (transitive) Anosov diffeomorphisms of class $C^1$ (instead of $C^{1+\delta}$), or even for (transitive) hyperbolic homeomorphisms (see e.g.\ \cite[Chapter~7]{Ruelle}, \cite[Part~II]{AY}, \cite[Chapter~11]{Akin}). 

We mention here the work \cite{CCZ} concerning the uniformity of the first Lyapunov exponent in a related setting.

\subsection{Relation with Sacker--Sell theory}\label{ss.SS}

Let $f$ be a homeomorphism of a compact metric space $X$. We assume that $f$ is \emph{invariantly connected}, i.e., the only invariant clopen sets are $\emptyset$ and $X$. This condition is satisfied if $f$ is transitive or $X$ is connected.

Let $\mathcal{A}$ be a continuous linear $d$-dimensional cocycle over $f$.
For each real number $\lambda$, we define a cocycle
\[
\mathcal{A}_\lambda (x,n) \coloneqq e^{-\lambda n} \mathcal{A}(x,n) \, .
\]
The \emph{Sacker--Sell spectrum} of $\mathcal{A}$ is the set 
\[
\mathrm{SS}(\mathcal{A}) \coloneqq 
\{ \lambda \in \R : \mathcal{A}_\lambda \text{ is not uniformly hyperbolic} \}.
\] 
(Here hyperbolic subbundles are allowed to be $\{0\}$.)
According to \cite[Theorem~2]{SS78} (see also \cite[Theorem~2.4]{JohnZamp}), the Sacker--Sell spectrum is a compact subset of $\R$ with at most $d$ connected components, that is, 
\begin{equation}\label{e.SS_intervals}
\mathrm{SS}(\mathcal{A}) = 
[\alpha_1,\beta_1] \cup \cdots \cup [\alpha_s,\beta_s] \, , \quad (s \le d).
\end{equation}
Furthermore, there exists an absolutely dominated splitting (see Section~\ref{ss.dominations} above)
into $s$ bundles,
\begin{equation}\label{e.SS_splitting}
\R^d = E_1(x) \oplus \cdots \oplus E_s(x) \, .
\end{equation}
As a consequence of domination (recall Proposition~\ref{p.DS_is_cont}), the subbundles vary continuously with $x$. 
As shown by Johnson, Palmer, and Sell \cite[Theorem~2.3]{JPS} the boundary points $\alpha_i$, $\beta_i$ of the connected components of the Sacker-Sell spectrum can be characterized in terms the Lyapunov exponents as follows. 
If we reindex the intervals in \eqref{e.SS_intervals} so that $\beta_1\ge\alpha_1 > \dots > \beta_s\ge\alpha_s$, then, for every $i$ we have
\begin{align}
\label{e.SS_beta}
\beta_i  &= \sup_{\mu} \tilde{\chi}_1(\mathcal{A}|E_i, \mu) \, , \\
\label{e.SS_alpha}
\alpha_i &= \inf_{\mu} \tilde{\chi}_{d_i}(\mathcal{A}|E_i, \mu) \, ,
\end{align}
where $\mu$ runs on the set of ergodic measures; furthermore, each sup and each inf are attained.
Using Proposition~\ref{p.SUSAET}, we can characterize these numbers as:
\begin{align*}
\beta_i  &= \lim_{n \to \infty} \sup_{x\in X} \frac{1}{n} \log \| \mathcal{A}(x,n)|E_i\| \, , \\
\alpha_i &= \lim_{n \to \infty} \inf_{x\in X} \frac{1}{n} \log m(\mathcal{A}(x,n)|E_i) \, .
\end{align*}

Note that the Sacker--Sell splitting \eqref{e.SS_splitting} is actually the \emph{finest} absolutely dominated splitting for the cocycle $\mathcal{A}$, that is, the absolutely dominated splitting into the maximal number of subbundles.

The Sacker--Sell spectrum is a discrete set if and only if $\alpha_i = \beta_i$ for each $i = 1,\dots,s$. Cocycles with discrete Sacker--Sell spectrum were studied by Johnson and Zampogni \cite{JohnZamp}. It turns out that this property coincides with complete regularity:

\begin{theorem}\label{t.equivalence}
A continuous cocycle $\mathcal{A}$ over an invariantly connected homeomorphism $f$ has discrete Sacker--Sell spectrum if and only if it is completely regular.
\end{theorem}

\begin{proof}
The fact that discrete Sacker--Sell spectrum implies complete regularity is the content of \cite[Proposition~3.1]{JohnZamp}; it also follows directly from the fundamental facts listed above. For the converse implication, if $\mathcal{A}$ is completely regular with exponents $\chi_1 > \cdots > \chi_s$, then by Corollary~\ref{c.DS}, its Oseledets splitting is the finest absolutely dominated splitting, and by relations \eqref{e.SS_beta}, \eqref{e.SS_alpha} we have $\beta_i = \chi_i = \alpha_i$ for each $i$, showing that $\mathcal{A}$ has discrete Sacker--Sell spectrum. 
\end{proof}

There is another classical object associated to a linear cocycle called the \emph{Mather spectrum} (see \cite{Mather}, \cite[\S~3.1]{Pesin}), which is directly related to the Sacker-Sell spectrum (see \cite[Corollary~6.45]{ChiLat}).

%%%%%%%%%%%%%%%%%%%%%%%%%%%%%%%%%%%%%%%%%%%%%%%%%%%%%%%%%%%%%%%%%%%%%%%%%%%%%%%%
\section{The Baire category of sets of regular points} \label{s.regular}
%%%%%%%%%%%%%%%%%%%%%%%%%%%%%%%%%%%%%%%%%%%%%%%%%%%%%%%%%%%%%%%%%%%%%%%%%%%%%%%%

\subsection{Baire category} 

Recall that a subset $A$ in a topological space $X$ is called \emph{nowhere dense} if its closure $\bar{A}$ has empty interior. 
A subset $A\subset X$ is said to be of the \emph{1\textsuperscript{st}  Baire category} if it can be represented as a union of a countable collection of nowhere dense sets, and of  \emph{2\textsuperscript{nd}  Baire category} otherwise. Sets of 1\textsuperscript{st}  Baire category are also called \emph{meager}. Note that any subset of a meager set is a meager set. 

A subset $A \subset X$ is called \emph{residual} if its complement is meager. 
If $X$ is a complete metric space, then residual sets are of 2\textsuperscript{nd} Baire category. 
In particular, $X$ and every other dense $G_\delta$ set are of the second category. 

Like the collection of sets of probability zero, the collection of meager sets contains the empty set, is closed under countable unions, is hereditary, and is proper. Therefore, one can think of meager sets as being ``small.'' See \cite{Oxtoby} for a full discussion.

\subsection{Dynamically defined meager sets} %{The Baire category of sets of convergence} 

We now describe a general mechanism for the occurrence of meager sets. 
Let $X$ be a complete metric space 
and $(\varphi_n)_{n>0}$ a sequence of real-valued continuous functions on $X$. Given numbers $\alpha,\beta\in\mathbb{R}$, define the following sets
\begin{align*}
I_\alpha	&\coloneqq	\big\{x\in X\colon\liminf\varphi_n(x)<\alpha\big\} \, ,\\
S_\beta		&\coloneqq	\big\{x\in X\colon\limsup\varphi_n(x)>\beta\big\} \, .
\end{align*}

\begin{theorem}\label{thm-seq-phin}
Assume that there are real numbers $\alpha<\beta$ such that the sets $I_\alpha$ and $S_\beta$ are both dense in $X$. Then the set of points $x\in X$ for which 
$\lim_{n\to\infty}\varphi_n(x)$ exists is meager. 
%% Actually, its complement is a residual set. \marginpar{\color{blue} Added ``actually'' part} 
\end{theorem}

\begin{proof} %[Proof of Theorem \ref{thm-seq-phin}]
Fix numbers $\alpha < \alpha' < \alpha'' < \beta'' < \beta' < \beta$.
Letting
\[
\tilde{I}_{\alpha'} \coloneqq \bigcap_{m=1}^\infty\bigcup_{n=m}^\infty\{x\in X\colon \varphi_n(x)<\alpha'\} \, ,
\]
we have $I_\alpha \subset \tilde{I}_{\alpha'} \subset I_{\alpha''}$.
The set $I_\alpha$ is dense, by assumption, and the set $\tilde{I}_{\alpha'}$ is a $G_\delta$.
Therefore the set $I_{\alpha''}$ contains a dense $G_\delta$, so it is residual. Similarly, $S_{\beta''}$ is residual. Thus, the set $I_{\alpha''} \cap S_{\beta''}$ is residual, and for any point in this set, $\lim_{n\to\infty}\varphi_n(x)$ does not exist.
\end{proof}

Without trying to be comprehensive, let us illustrate how Theorem~\ref{thm-seq-phin} can be applied to a simple dynamical setting. Suppose that $f : X \to X$ is a continuous transformation of a compact metric space. Given a (say) continuous function $\varphi : X \to \R$, consider the sequence of Birkhoff averages
\begin{equation*} %\label{e.func-phin}
\varphi_n(x) \coloneqq \frac{1}{n}\sum_{k=0}^{n-1}\varphi(f^k(x)) 
\end{equation*}
and let $\mathcal{R}_\varphi$ be the set of points $x\in X$ for which $\lim_{n\to\infty}\varphi_n(x)$ exists. 
By Birkhoff ergodic theorem, we know that this set has full measure with respect to any invariant probability measure and hence, is ``large'' from the measure-theoretical point of view.
But what can we say about the Baire category of this set?  

Following \cite{AAN}, a (perhaps noninvertible) continuous transformation $f$ of a compact metric space $X$ is called \emph{strongly transitive} if the \emph{backward orbit} $ \{f^{-n}(x) \colon n \ge 0\}$ of every point $x \in X$ is dense in $X$. 
Examples include any minimal homeomorphism, any transitive one-sided subshift of finite type, and the restriction of any rational map of $\bar{\mathbb{C}}$ to its Julia set (see \cite[Corollary 4.13]{Milnor}).

\begin{theorem}\label{t.dichotomy_Birkhoff}
Let $f \colon X \to X$ be a strongly transitive map, and let $\varphi \colon X \to \R$ be a continuous function. 
Then the set $\mathcal{R}_\varphi$ of points of convergence of Birkhoff averages is either a meager set or the whole space $X$. In the latter case, the Birkhoff averages of $\varphi$ with respect to $f$ converge uniformly to a constant. 
\end{theorem}

\begin{proof}
Suppose that the Birkhoff averages do not converge uniformly to a constant. By Corollary~\ref{c.maybe_Oxtoby}, there exist two ergodic Borel probability measures $\mu_1$ and $\mu_2$ such that $\int \varphi \, d\mu_1 \neq \int \varphi \, d\mu_2$. For $i=1,2$, let $\mathcal{B}_i$ be the set of points $x$ such that $\varphi_n(x) \to \int \varphi \, d\mu_i$. These sets are nonempty and invariant. In particular, they contain full backward orbits. It follows from strong transitivity that both sets $\mathcal{B}_1$ and $\mathcal{B}_2$ are dense in $X$. It follows from Theorem~\ref{thm-seq-phin} that the set of convergence $\mathcal{R}_\varphi$ is meager.
\end{proof}

The literature contains many related results; we mention \cite[Theorem~4.3.(2)]{John78}, \cite[Theorem~11, p.~166]{Akin}, and \cite{CaVar}. The general message is that the set of points for which Birkhoff ergodic theorem holds, while large from a measure-theoretic viewpoint, is often meager. 

As it is reasonable to expect, similar phenomena for the multiplicative ergodic theorem of Oseledets occur in a variety of contexts: see for instance \cite[Theorem~3.14]{ABC}, \cite{Tian}, \cite{MdL}, \cite{Varandas}, and \cite[Corollary~C]{LinTian}. We now turn to our investigation of this matter, where the previously discussed concept of complete regularity will naturally enter into the picture. Our goal is to establish sufficient conditions under which the set $\mathcal{R}$ of LP-regular points is meager, thereby ensuring that the set of irregular points is topologically large. We note that the set of irregular points may also be large in terms of entropy \cite{CZZ} or Lebesgue measure \cite{KLNS}.

\subsection{A criterion for generic irregularity}

We now come back to the setting of continuous linear cocycles. 
We will formulate a sufficient condition for the set of LP-regular points to be meager. 
Our condition uses the notion of complete regularity studied above. 

Given a homeomorphism $f$ of a compact metric space $X$ and a nonempty subset  $Y\subset X$, we define the \emph{stable set} of $Y$ as
$$
W^\mathrm{s}(Y) = \big\{x\in X: d(f^n(x),Y)\to0 \text{ as } n\to+\infty \big\} \, .
$$ 
Now we can state the following result:

\begin{theorem}\label{thm-irreg-criterion}
Let $f$ be a homeomorphism of a compact metric space $X$ and let $\mathcal{A}$ be a continuous linear cocycle over $f$. 
Assume that there exist two $f$-invariant disjoint compact subsets $X_1, X_2\subset X$ such that
\begin{enumerate}
\item the stable sets $W^\mathrm{s}(X_1)$ and $W^\mathrm{s}(X_2)$ are both dense in $X$;
\item the restrictions of the cocycle $\mathcal{A}$ to $X_1$ and $X_2$, denoted $\mathcal{A}|X_1$ and $\mathcal{A}|X_2$, are both completely regular;
\item the cocycles $\mathcal{A}|X_1$ and $\mathcal{A}|X_2$ have different Lyapunov spectra, that is, $\mathrm{Sp}(\mathcal{A}|X_1)\ne\mathrm{Sp}(\mathcal{A}|X_2)$.
\end{enumerate} 
Then the set $\mathcal{FR}$  of points which are forward regular for the cocycle $\mathcal{A}$ is meager and, in particular, so is the set $\mathcal{R}$ of LP-regular points.
\end{theorem}

The proof of this theorem will use Theorem~\ref{thm-seq-phin} above and also the following lemma:

\begin{lemma}\label{lemma:7.1}
Let $Y\subset X$ be a compact $f$-invariant subset such that the restriction of the cocycle $\mathcal{A}$ to $Y$ is completely regular with Lyapunov exponents $\chi_1>\cdots>\chi_s$. 
Then, for every $x\in W^\mathrm{s}(Y)$,
\begin{equation}\label{e.phin111}
\lim_{n\to\infty}\frac1n\log\|\mathcal{A}(x,n)\|=\chi_1 \, .
\end{equation}
\end{lemma}

\begin{proof}
It follows from the definition of complete regularity and property \eqref{e.oseledets_norm} that the desired conclusion \eqref{e.phin111} holds for every point in $Y$. 
Actually, it follows from Theorem~\ref{t.uniformity} that the limit in \eqref{e.phin111} is uniform over $Y$.
In particular, given $\varepsilon>0$, we can find $N>0$ such that, for all $y \in Y$, 
\[
\frac{1}{N} \log\|\mathcal{A}(y,N)\| < \chi_1 + \varepsilon \, .
\] 
By continuity, if $d(x,Y)$ is sufficiently small, then 
\[
\frac{1}{N} \log\|\mathcal{A}(x,N)\| < \chi_1 + 2\varepsilon \, .
\] 
Then, as a straightforward consequence of submultiplicativity of norms, we have
\[
\limsup_{n \to \infty }\frac{1}{n} \log\|\mathcal{A}(x,n)\| \le \chi_1 + 2\varepsilon 
\]
for all $x \in W^\mathrm{s}(Y)$.
Since $\varepsilon>0$ is arbitrary, the limsup above is actually $\le \chi_1$ for all $x \in W^\mathrm{s}(Y)$.

In order to obtain an inequality in the reverse direction, 
fix $\varepsilon>0$ again.
Using Theorem~\ref{t.uniformity} we find $N>0$ such that, for all points $y \in Y$ and all unit vectors $u \in E_1(y)$,
\begin{equation}\label{e.point_vector}
\frac{1}{N} \log\|\mathcal{A}(y,N) u\| > \chi_1 - \varepsilon \, .
\end{equation}

Since the cocycle $\mathcal{A}$ restricted to the compact invariant set $Y$ admits a dominated splitting with dominating bundle $E_1$, we can define a family of cones $C(y) \subset \R^d$ depending continuously on $y \in Y$ with the following properties: for every $y \in Y$, we have $E_1(y) \subset C(y)$ and 
\[
\overline{\mathcal{A}(y,1)(C(y))} \subset \mathrm{int}(C(y))  \cup \{0\} \quad \text{(strict invariance).}
\]
Such a cone field is easily constructed using an adapted metric: see \cite{Gourmelon,CroPo}.
The subbundle attracts vectors in the cone field in the following sense:
for every $y \in Y$ and every unit vector $u \in C(y)$, 
\begin{equation}\label{e.attract}
\measuredangle \big( \mathcal{A}(y,n)u , E_1(f^n(y)) \big) \to 0 \text{ as } n \to \infty, 
\end{equation}
and this convergence is actually uniform with respect to $y$ and $u$.
We extend continuously the cone field to a neighborhood $U$ of $Y$ so that the strict invariance property is satisfied for points in $U \cap f^{-1}(U)$.

Fix $x\in W^\mathrm{s}(Y)$. 
Replacing $x$ by an iterate if necessary, we assume that $f^n(x) \in U$ for every 
$n \ge 0$. 
Fix any unit vector $v \in C(x)$.
It follows from property \eqref{e.attract} that there exist a sequence of points $y_n\in Y$ and a sequence of unit vectors $u_n \in E_1(y_n)$ such that 
\[
d(f^n(x), y_n)\to 0 \quad \text{and} \quad
\measuredangle(\mathcal{A}(x,n)v , u_n) \to 0 \quad \text{as } n \to +\infty \, .
\]
Inequality \eqref{e.point_vector} applies to each pair $(y,u) = (y_n, u_n)$.
For every sufficiently large $n$, the pair $\big(f^n(x),\frac{\mathcal{A}(x,n)v}{\|\mathcal{A}(x,n)v\|}\big)$ is close to $(y_n,u_n)$, and therefore 
\[
\frac{1}{N} \log \frac{\|\mathcal{A}(x, n+N) v \|}{\|\mathcal{A}(x, n) v \|} > \chi_1 - 2\varepsilon  \, .
\]
Now it is straightforward to check that
\[
\liminf_{n \to \infty }\frac{1}{n} \log\|\mathcal{A}(x,n) v\| \ge \chi_1 - 2\varepsilon  \, .
\]
Since $\varepsilon>0$ is arbitrary, the liminf above is actually $\ge \chi_1$ for all $x \in W^\mathrm{s}(Y)$.
The proof of the lemma is completed.
\end{proof}

\begin{proof}[Proof of Theorem \ref{thm-irreg-criterion}]
For each $j=1,2$, let $\tilde\chi_1^j\ge\cdots\ge\tilde\chi_d^j$ be the list of values of the Lyapunov exponent of the cocycle $\mathcal{A}|X_j$, where each value is repeated according to multiplicity.
By assumption, these two lists are not identical.

We first assume that $\tilde\chi_1^1\neq\tilde\chi_1^2$. 
For definiteness, assume that $\tilde\chi_1^1 > \tilde\chi_1^2$.
Consider the sequence of continuous functions 
\begin{equation}\label{e.sec-func-phin1}
\varphi_n(x)=\frac1n\log\|\mathcal{A}(x,n)\|, \quad x\in X \, .
\end{equation}
By Lemma \ref{lemma:7.1}, for each $j=1,2$, we have $\lim \varphi_n(x) = \tilde\chi_1^j$ for all $x \in W^\mathrm{s}(X_j)$. Each of these stable sets is dense, by hypothesis. 
Applying Theorem~\ref{thm-seq-phin} with $\alpha=\tilde\chi_1^2$ and $\beta=\tilde\chi_1^1$, we obtain that the set of points for which $\lim_{n\to\infty}\varphi_n(x)$ exists is meager and hence, so is its subset $\mathcal{FR}$ of forward regular points (see \cite{BP}). This completes the proof of the theorem in the particular case when $\tilde\chi_1^1 \neq \tilde\chi_1^2$. 

We now proceed with the general case and show how to reduce it to the previous one by using exterior powers. 
As explained in Subsection~\ref{ss.exterior}, a regular point $x\in X$ for the cocycle $\mathcal{A}$ is also a regular point for the exterior power cocycle $\mathcal{A}^{\wedge i}$. Now let $i$ be the smallest index such that  $\tilde\chi_i^1\neq\tilde\chi_i^2$, that is, $\tilde\chi_i(\mathcal{A}|X_1)\ne\tilde\chi_i(\mathcal{A}|X_2)$. Then, by formula \eqref{e.power-top}, we obtain that 
$$
\tilde\chi_1\big(\mathcal{A}^{\wedge i}|X_1\big) \ne \tilde\chi_1\big(\mathcal{A}^{\wedge i}|X_2\big) \, .
$$
Since the cocycle $\mathcal{A}$ is completely regular on subsets $X_1$ and $X_2$,  
the exterior power cocycle $\mathcal{A}^{\wedge i}$ is also completely regular on these sets (see Proposition~\ref{p.ep}). By the previous case, the set of forward regular points for the cocycle 
$\mathcal{A}^{\wedge i}$ is meager, and hence, so is the set of forward regular points for the cocycle 
$\mathcal{A}$. This completes the proof of the theorem.
\end{proof}

\subsection{Applications of the criterion} 

We present some corollaries of Theorem~\ref{thm-irreg-criterion}.

If $p$ is a periodic point for $f$, let $O(p)$ denote the orbit of $p$.
Given any linear cocycle $\mathcal{A}$ over $f$, the periodic point $p$ is LP-regular, and its Lyapunov exponents can be computed from the eigenvalues of the matrix $\mathcal{A}(p,k)$, where $k$ is the period of $p$.

\begin{corollary}\label{cor-criterion}
Let $f$ be a homeomorphism of a compact metric space $X$ and let $\mathcal{A}$ be a continuous linear cocycle over $f$. 
Suppose there are two periodic points $p,q\in X$ such that 
\begin{enumerate}
\item the stable sets $W^\mathrm{s}(O(p))$ and $W^\mathrm{s}(O(q))$ are both dense in $X$;
\item $\mathrm{Sp}(\mathcal{A},p)\ne\mathrm{Sp}(\mathcal{A},q)$.
\end{enumerate} 
Then the set $\mathcal{FR}$ of points which are forward regular for the cocycle $\mathcal{A}$ is meager and, in particular, so is the set $\mathcal{R}$ of LP-regular points.
\end{corollary}

\begin{proof} 
Recall from Example~\ref{ex1} that the restriction of the cocycle $\mathcal{A}$ to any periodic orbit is completely regular. The result now follows from Theorem~\ref{thm-irreg-criterion}.
\end{proof}

There are situations where Theorem~\ref{thm-irreg-criterion} applies, but Corollary~\ref{cor-criterion} does not. For example, suppose $\varphi^t \colon M \to M$ is a transitive Anosov flow of a compact manifold. If $t_0>0$ is small enough, then $f \coloneqq \varphi^{t_0}$ has no periodic points, so we are out of the scope of Corollary~\ref{cor-criterion}. Now, suppose $x_1$ and $x_2$ are periodic points for the flow, and let $X_1$ and $X_2$ be the respective orbits. Then the restrictions $f|X_1$ and $f|X_2$ are conjugate to irrational rotations. If $\mathcal{A}$ is any continuous cocycle over $f$ whose restrictions $\mathcal{A}|X_1$ and $\mathcal{A}|X_2$ are completely regular but have different Lyapunov spectra, then its set $\mathcal{R}$ of LP-regular points is meager in $M$. Indeed, the stable sets $W^\mathrm{s}(X_1)$ and $W^\mathrm{s}(X_2)$ are dense (see e.g.\ \cite[Theorem~2.6.10]{FisherH}), so Theorem~\ref{thm-irreg-criterion} applies.

\medskip

On the other hand, there is a particular setting of the dynamics in which the denseness hypothesis in Corollary~\ref{cor-criterion} is naturally satisfied, namely homoclinic classes. We recall the definition.

Let $f: M\to M$ be a $C^1$ diffeomorphism of a compact smooth Riemannian manifold $M$. Let also $p,q\in M$ be two hyperbolic periodic points. 
We say that $p$ and $q$ are \emph{homoclinically related} and write $q\sim p$ if the sets of transverse intersection points $W^\mathrm{u}(O(p))\pitchfork W^\mathrm{s}(O(q))$ and $W^\mathrm{s}(O(p))\pitchfork W^\mathrm{u}(O(q))$ are both nonempty. This turns out to be an equivalence relation (see \cite[Proposition~2.1]{Newhouse}), and every equivalence class is an invariant subset. The \emph{homoclinic class} of $p$ is defined to be the closure of the equivalence class of $p$. (We note that two homoclinic classes may intersect without being identical.) 
Furthermore (see \cite[Corollary~2.7]{Newhouse}), 
\begin{equation}\label{Newhouse}
H_p=\overline{W^\mathrm{u}(O(p))\pitchfork W^\mathrm{s}(O(p))}.
\end{equation}
In particular, for any periodic point $q \sim p$ the set $W^\mathrm{s}(O(q))$ is dense in $H_p = H_q$. Thus Corollary~\ref{cor-criterion} implies the following result:

\begin{corollary}\label{cor1}
Let $f$ be a $C^1$ diffeomorphism, let $p$ be a hyperbolic periodic point, and let $H_p$ be its homoclinic class.
Let $\mathcal{A}$ be any continuous cocycle on $H_p$.
Assume that there exists another periodic point $q$ which is homoclinically related to $p$ such that $\mathrm{Sp}(\mathcal{A}|O(p))\ne\mathrm{Sp}(\mathcal{A}|O(q))$. Then the set $\mathcal{R}$ of LP-regular points for $\mathcal{A}$ is meager in the relative topology of $H_p$.
\end{corollary}

Of course, the most natural application of the result above is for the derivative cocycle restricted to the homoclinic class.

\medskip

For the sake of completeness, let us describe here a related result by Tian \cite{Tian}.
Let $f$ be a homeomorphism of a compact metric space $X$. We say that $f$ has \emph{exponential specification property with exponent $\lambda> 0$} if for any $\delta> 0$ there is $N=N(\delta) > 0$ such that for any $k\ge 1$, any $k$ points $x_1,x_2,\dots,x_k\in X$, any integers $a_1\le b_1 <a_2\le b_2 <\ldots<a_k\le b_k$ with $a_{i+1}-b_i\ge N$, one can find a point $y\in X$ such that 
$$
d(f^i(y),f^i(x_j))<\delta e^{-\lambda\min\{i-a_i,b_j-i\}} 
$$
for $a_j\le i\le b_j$ and $1\le j\le k$.

\begin{theorem}[{see \cite[Theorem 1.4]{Tian}}] \label{t.Tian}
Let $\mathcal{A}$ be a H\"older continuous cocycle over a homeomorphism $f$ with the exponential specification property. Then 
either the set $\mathcal{R}$ of LP-regular points is meager in $X$
or the cocycle has the same Lyapunov spectrum with respect to all ergodic $f$-invariant measures. 
\end{theorem}

A homeomorphism $f$ with the exponential specification property may fail to have periodic orbits (see \cite[Example~2.4.(2)]{Tian}). Thus we cannot use Corollary~\ref{cor-criterion} to derive Theorem~\ref{t.Tian}.
Note that Theorem~\ref{thm-irreg-criterion} and its corollaries apply to all continuous cocycles. In contrast, Theorem~\ref{t.Tian} requires H\"older continuity of the cocycle, since its proof employs estimates from \cite{Kalinin} that rely crucially on this stronger regularity assumption.

\subsection{All-or-meager results}

Here is a related result, this time following from Corollary~\ref{cor-criterion} combined with Theorem~\ref{t.amenable_reduction}:

\begin{theorem}\label{t.dichotomy}
Let $\mathcal{A}$ be a H\"older continuous cocycle over a transitive Anosov $C^{1+\delta}$ diffeomorphism $f$ of a compact smooth manifold $X$.
Then either the set $\mathcal{R}$ of LP-regular points is meager in $X$
or $\mathcal{R} = X$ and the cocycle is completely regular.
\end{theorem}

Notice the sharp contrast between the two alternatives: either the set $\mathcal{R}$ is the whole space, or it is a meager subset. We call statements with this type of conclusion \emph{all-or-meager} results. Another example is Theorem~\ref{t.dichotomy_Birkhoff}.

Let us look for other results of this type.
Optimistically, one could ask: are the conclusions of Theorem~\ref{t.dichotomy} valid for any continuous linear cocycle over any homeomorphism of a compact space? The answer is negative, as the following example shows.

\begin{example}
Let $g : Y \to Y$ be a homeomorphism of a compact metric space $Y$ and let $\mathcal{B}$ be a linear cocycle over $g$, with generator $B$.
We assume two conditions: 
First, $g$ admits a fully-supported Borel probability measure $\mu$.
Second, the set $\mathcal{R}_{\mathcal{B}}$ of LP-regular points is not the whole space $Y$. 
Examples satisfying these conditions can be found using Theorem~\ref{t.dichotomy}, for instance. 
Note that $\mu$-almost every point in~$Y$ has alpha- and omega-limit sets both equal to $Y$. 
Fix a point $y_0$ with these properties which is also an element of the full-measure set $\mathcal{R}_\mathcal{B}$.
Let $y_n \coloneqq g^n(y_0)$; then, for any $N>0$, the sets $\{y_n : n > N\}$ and $\{y_n : n < -N\}$ are both dense in~$Y$. Now fix an increasing two-sided sequence $(t_n)_{n \in \Z}$ such that $t_n \to 0$ as $n \to -\infty$ and $t_n \to 1$ as $n \to +\infty$.  
Define the following subsets of $Y \times [0,1]$:
\[
Z \coloneqq \{(y_n, t_n) : n \in \Z \} \quad \text{and} \quad 
X \coloneqq (Y \times \{0,1\}) \cup Z \, .
\]
Then $\overline{Z} = X$ and in particular $X$ is compact.
Also note that each point in $Z$ has a neighborhood disjoint from $Y \times \{0,1\}$ and in particular $Z$ is a relatively open subset of $X$.
Define a map $f \colon X \to X$ as follows:
\[
f(y,t) \coloneqq (g(y),t) \text{ if } t=0,1, \quad
f(y_n,t_n) \coloneqq (y_{n+1},t_{n+1}) \, .
\]
Then $f$ is actually a homeomorphism of $X$, and the subset $Z$ is the orbit of $(y_0,t_0)$.
Finally, define a cocycle $\mathcal{A}$ over $f$ whose generator is $\mathcal{A}(y,t) \coloneqq B(y)$.
Then the set $\mathcal{R}_\mathcal{A}$ of LP-regular points is not the whole space $X$, as its complement includes $(Y\setminus \mathcal{R}_\mathcal{B}) \times \{0,1\}$. 
On the other hand, since the set $\mathcal{R}_\mathcal{B}$ contains the point $y_0$, the set $\mathcal{R}_\mathcal{A}$ contains the point $(y_0,t_0)$, as well as its orbit $Z$.
Since $Z$ is an open and dense subset of $X$, it follows that $\mathcal{R}_\mathcal{A}$ is not a meager set.
Thus, we have exhibited an example where the all-or-meager dichotomy fails.
\end{example}

The example above is somewhat unsatisfactory because the map $f$ has wandering points. 

Returning to the case of transitive Anosov diffeomorphisms, it is an open problem whether Theorem~\ref{t.dichotomy} extends to cocycles that are merely continuous rather than Hölder continuous.

On the other hand, if we assume that the dynamics is \emph{minimal}, then we have the following all-or-meager dichotomy that applies to all continuous linear cocycles:

\begin{theorem}\label{t.dichotomy_minimal}
Let $\mathcal{A}$ be a continuous cocycle over a minimal homeomorphism~$f$ of a compact metric space.
Then either the set $\mathcal{R}$ of LP-regular points is meager in $X$
or $\mathcal{R} = X$ and the cocycle is completely regular.
\end{theorem}

The proof of this theorem is given in Section~\ref{ss.dichotomy_proof}.

Theorem~\ref{t.dichotomy_minimal} extends some previous results such as \cite[Theorem~4]{Furman} and \cite[Corollary~6.6]{HLT}. 

Johnson and Zampogni asked a question in \cite[p.~104]{JohnZamp}, which in our terminology is as follows: Assume that the dynamics $f$ is minimal, and that for all points $x$ and all nonzero vectors $v$, the forward Lyapunov exponent $\lim_{n \to \infty} \frac1n \log\|\mathcal{A}(x,n)v\|$ is well-defined; does it follow that the cocycle is completely regular? Note that Theorem~\ref{t.dichotomy_minimal} provides a positive answer to this question under the stronger hypothesis that all points are LP-regular.

\subsection{Singular values of cocycles}

As a preparation for the proof of Theorem~\ref{t.dichotomy_minimal}, we discuss the behaviour of the singular values of the cocycle. Furthermore, under assumption of minimality, we obtain new information about singular values in the absence of dominated splitting (Theorem~\ref{t.zero_liminf}).

The \emph{singular values} of a linear operator $L \colon \R^d \to \R^d$ are the eigenvalues of $\sqrt{L^* L}$, which we list with multiplicity as 
\[
\sigma_1(L) \ge \cdots \ge \sigma_d(L) \, .
\] 
Each singular value $\sigma_k(L)$ depends continuously on $L$.
The extreme singular values are the norm $\|L\| = \sigma_1(L)$ and the co-norm $m(L) = \sigma_d(L)$. 
More generally, by the \emph{Courant--Fischer principle} \cite[Theorem~7.3.8]{HornJ}, for any $k = 1, \dots, d$,
\begin{alignat*}{2}
\sigma_k(L)
&= \max \big\{ m(A|_E) &\colon &\text{$E$ is a subspace of dim.\ $k$} \big\}  \\
&= \min \big\{ \|A|_F\| &\colon &\text{$F$ is a subspace of dim.\ $d-k+1$} \big\} \, . 
\end{alignat*}
As a simple consequence, we have the following useful bounds:
% See also T. Tao, Topics in Random Matrix Theory, exercises 1.3.21 and 24. 
\begin{equation}\label{e.BPS}
m(L_3) \, \sigma_k(L_2) \, m(L_1) \le 
\sigma_k(L_3 L_2 L_1) \le 
\|L_3\| \, \sigma_k(L_2) \, \|L_1\| \, .
\end{equation}

The list of singular values of the exterior power $L^{\wedge r}$ coincides with the list $\sigma_{k_1}(L) + \cdots + \sigma_{k_r}(L)$, where $1 \le k_1 < \cdots < k_r \le d$, with repetitions being taken into account. 

The Lyapunov exponents of a linear cocycle $\mathcal{A}$ are related to the singular values as follows: for every forward regular point $x$, 
\[
\tilde\chi_k(x) = \lim_{n \to +\infty} \frac{1}{n} \log \sigma_k(\mathcal{A}(x,n)) \, .
\]
% see e.g.\ \cite[\S~3.3.2]{Ar}.

Dominated splittings (see Section~\ref{ss.dominations})
can be detected in terms of separation of singular values as follows:

\begin{proposition}[\!\!{\cite[Theorem~A]{BochiGourmelon}}]\label{p.BG}
Let $f$ be a homeomorphism of a compact metric space, and let $\mathcal{A}$ be a $d$-dimensional continuous linear cocycle over $f$.
For every $k$ with $0<k<d$, the cocycle $\mathcal{A}$ admits a dominated splitting with a nontrivial dominating bundle of dimension $k$ if and only if there exist positive constants $c$ and $\varepsilon$ such that
\[
\frac{\sigma_{k}(\mathcal{A}(x,n))}{\sigma_{k+1}(\mathcal{A}(x,n)) } 
> c e^{\varepsilon n}
\]
for all $x \in X$ and all $n>0$.
\end{proposition}

In the result above, the ``only if'' part is rather trivial. 
It is the ``if'' part that is interesting: the exponential separation of singular values forces the existence of a dominated splitting. 

Our next result can be seen as an improvement of Proposition~\ref{p.BG} under the assumption of minimality:

\begin{theorem}\label{t.zero_liminf}
Let $f$ be a minimal homeomorphism of a compact metric space, and let $\mathcal{A}$ be a $d$-dimensional continuous linear cocycle over $f$.
Fix $k$ with $0<k<d$.
Then the cocycle $\mathcal{A}$ does \emph{not} admits a dominated splitting with a dominating bundle of dimension $k$ if and only if
there exists a residual subset $G$ of $X$ such that for every $x \in G$, 
\[
\liminf_{n \to \infty}\frac{1}{n} \log \frac{\sigma_k(\mathcal{A}(x, n))}{\sigma_{k+1}(\mathcal{A}(x, n))}  = 0 \, .
\]
\end{theorem}

Note that the minimality assumption cannot be dropped, as Example~\ref{ex.tangency} shows.

\begin{proof}[Proof of Theorem~\ref{t.zero_liminf}]
The ``if'' part is a direct consequence of (the trivial half of) Proposition~\ref{p.BG}.
So let us prove the ``only if'' part.
Assume that the cocycle $\mathcal{A}$ admits no dominated splitting with dominating bundle of dimension $k$.
By Proposition~\ref{p.BG}, for every $i>0$ there exists $x_i \in X$ and $n_i > i$ such that
\[
\frac{1}{n_i} \log \frac{\sigma_{k}(\mathcal{A}(x_i,n_i))}{\sigma_{k+1}(\mathcal{A}(x_i,n_i))}  < \frac{1}{i} \, .
\]
Since $X$ is compact and the cocycle is continuous, there exists $c>0$ such that
\[
\|\mathcal{A}(x,n)\| \le e^{c|n|}
\]
for all $x\in X$ and $n \in \Z$.
Consider the matrix identity
\[
\mathcal{A}(f^j(x), n) = 
\mathcal{A}(f^{n}(x), j)  \, \mathcal{A}(x, n)  \, \mathcal{A}(f^j(x), -j) \, .
\]
Using inequalities \eqref{e.BPS}, 
\[
\frac{\sigma_k(\mathcal{A}(f^j(x), n))}{\sigma_{k+1}(\mathcal{A}(f^j(x), n))} \le 
e^{4c|j|} \,
\frac{\sigma_k(\mathcal{A}(x,n))}{\sigma_{k+1}(\mathcal{A}(x, n))} \, .
\]
In particular, if $|j|< \sqrt{n_i}$, then
\[
\frac{1}{n_i} \log \frac{\sigma_k(\mathcal{A}(f^j(x_i), n_i))}{\sigma_{k+1}(\mathcal{A}(f^j(x_i), n_i))} < \frac{4c}{\sqrt{n_i}} + \frac{1}{i} \eqcolon \varepsilon_i \, .
\]
Note that $\varepsilon_i \to 0$ as $i \to \infty$.

Now define a sequence of open sets
\[
U_i \coloneqq \left\{ x \in X \colon 
\frac{1}{n_i} \log \frac{\sigma_k(\mathcal{A}(x, n_i))}{\sigma_{k+1}(\mathcal{A}(x, n_i))} <  \varepsilon_i 
\right\} \, .
\]
By the previous observations, $U_i$ contains the orbit segment $\{f^j(x_i) : |j| < \sqrt{n_i}\}$. Since $n_i \to \infty$ and $f$ is a minimal transformation, there exists a sequence $\delta_i \to 0$ such that each set $U_i$ is $\delta_i$-dense in $X$.
It follows that, for each $m>0$, the set $V_m \coloneqq \bigcup_{i=m}^\infty U_i$ is open and dense. By Baire's theorem, the set $G \coloneqq \bigcap_{m=1}^\infty V_m$ is residual, and in particular dense.

For every point $x$ in $G$, there exists a sequence $i_j \to \infty$ such that $x \in U_{i_j}$, and so
\[
\lim_{j \to \infty}\frac{1}{n_{i_j}} \log \frac{\sigma_k(\mathcal{A}(x, n_{i_j}))}{\sigma_{k+1}(\mathcal{A}(x, n_{i_j}))}  = 0 \, . \qedhere
\]
\end{proof}

Let us highlight a consequence of Theorem~\ref{t.zero_liminf} in the particular context discussed in Example~\ref{ex3}, Section~\ref{ss.examples}:

\begin{corollary}
Let $f$ be a uniquely ergodic homeomorphism of a compact metric space and let $\mathcal{A}$ be a continuous $\mathrm{SL}^{\pm}(2,\R)$-cocycle over $f$. If the cocycle $\mathcal{A}$ is not uniformly hyperbolic, then there exists a residual subset $G_0$ of $X$ such that for all $x\in G_0$, the set of accumulation points of the sequence $\frac{1}{n}\log\|\mathcal{A}(x,n)\|$ is the interval $[0,\chi_1]$, where $\chi_1 = \chi_1(\mu)$ is the first Lyapunov exponent of the cocycle with respect to the unique $f$-invariant measure $\mu$.
\end{corollary}

Note that the result above improves some of the conclusions of \cite[Theorem~4]{Furman}.

\begin{proof}
For each $x$, let $\Lambda_x$ be the set of accumulation points of the sequence $\frac{1}{n}\log\|\mathcal{A}(x,n)\|$. Then $\Lambda_x$ is a compact interval (see e.g.\ \cite[Lemma~2.5]{Walters}).
Note that $\Lambda_x \subset [0,\infty)$ because $|\det \mathcal{A}| \equiv 1$, and $\Lambda_x \subset [0,\chi_1]$ due to Proposition~\ref{p.SUSAET}.
If $\chi_1 = 0$, there is nothing to show.
So, assume that $\chi_1$ is positive.
Since the cocycle is not uniformly hyperbolic, by Theorem~\ref{t.zero_liminf} there is a dense set of points $x$ for which $\liminf \frac{1}{n}\log\|\mathcal{A}(x,n)\| = 0$. On the other hand, there is another dense set of points $x$ for which $\lim \frac{1}{n}\log\|\mathcal{A}(x,n)\| = \chi_1$. The proof of Theorem~\ref{thm-seq-phin} actually shows that if $0<\alpha''<\beta''<\chi_1$, then there is a residual subset of points $x$ such that 
\[
\liminf \frac{1}{n}\log\|\mathcal{A}(x,n)\| < \alpha'' < \beta'' < \limsup \frac{1}{n}\log\|\mathcal{A}(x,n)\| 
\] 
and in particular $\Lambda_x \supset [\alpha'', \beta'']$. 
Therefore, the set of points $x$ such that $\Lambda_x = [0,\chi_1]$ is residual. 
\end{proof}

\subsection{Proof of Theorem~\ref{t.dichotomy_minimal}}\label{ss.dichotomy_proof}

Recall that, given a ergodic measure $\mu$, its Lyapunov exponents, repeated according to their multiplicities, are denoted $\tilde{\chi}_1(\mathcal{A},\mu) \ge \cdots \ge \tilde\chi_d(\mathcal{A},\mu)$.

\begin{lemma}\label{l:minimal_no_dom}
Let $f$ be a minimal homeomorphism of a compact metric space~$X$.
Let $\mathcal{A}$ be a $d$-dimensional continuous linear cocycle over $f$ and let $\mathcal{R}$ be the set of its LP-regular points.
Suppose that $\mathcal{A}$ admits no dominated splitting into non-trivial subbundles. 
% except for the ``trivial'' dominated splitting consisting of a single nonzero subbundle.
Then:
\begin{enumerate}
\item\label{i.equal_LE}
either $\tilde\chi_1(\mathcal{A},\mu) = \tilde\chi_d(\mathcal{A},\mu)$ for every $f$-ergodic Borel probability measure $\mu$, in which case the set $\mathcal{R}$ equals $X$ and the cocycle $\mathcal{A}$ is completely regular;
\item\label{i.different_LE} 
or $\tilde\chi_1(\mathcal{A},\mu) > \tilde\chi_d(\mathcal{A},\mu)$ for some $f$-ergodic Borel probability measure $\mu$, in which case the set $\mathcal{R}$ is meager in $X$.
\end{enumerate}
\end{lemma}

\begin{proof}
As a first case, suppose that $\tilde\chi_1(\mathcal{A},\mu) = \tilde\chi_d(\mathcal{A},\mu)$ for every $f$-ergodic Borel probability measure $\mu$.
Consider the following ``normalized'' cocycle
\begin{equation}\label{e.normalized}
\hat{\mathcal{A}}(x,n) \coloneqq \left| \det \mathcal{A}(x,n) \right|^{-\frac{1}{d}} \mathcal{A}(x,n) \, ,
\end{equation}
whose determinants are $\pm 1$.
Our assumption ensures that all Lyapunov exponents of $\hat{\mathcal{A}}$ with respect to all ergodic measures are zero. By Proposition~\ref{p.one_exponent}, the cocycle $\hat{\mathcal{A}}$ is completely regular with zero spectrum. 

Let $\varphi_n(x) \coloneqq \log |\det \mathcal{A}(x,n)|$, and let $\mathcal{D}$ be the set of $x \in X$ such that the following two limits exist and coincide:
\begin{equation}\label{e.two_limits}
\lim_{n\to \infty} \frac{\varphi_n(x)}{n} = \lim_{n\to -\infty} \frac{\varphi_n(x)}{n} \, .
\end{equation}
We claim that $\mathcal{R}$ (the set of LP-regular points for $\mathcal{A}$) equals $\mathcal{D}$. Indeed, the inclusion $\mathcal{R} \subset \mathcal{D}$ always holds, as for LP-regular points the limit \eqref{e.two_limits} will be the sum of the Lyaponov exponents. On the other hand, since $\hat{\mathcal{A}}$ is completely regular with zero exponents, it follows from \eqref{e.normalized} that
\begin{equation}\label{e.backyard}
\log \|\mathcal{A}(x,n) v\| - \frac{\varphi_n(x)}{d} = o(|n|) 
\end{equation}
for all $x \in X$ and all nonzero $v \in \R^d$, showing the reverse inclusion $\mathcal{D} \subset \mathcal{R}$.

Applying Theorem~\ref{t.dichotomy_Birkhoff} to $f$ and $f^{-1}$, we see that the set $\mathcal{D}$ (that is, $\mathcal{R}$) is either meager or equal to $X$ and, in the latter case, the limits \eqref{e.two_limits} equal some constant $\chi$. So, when $\mathcal{R}=X$, \eqref{e.backyard} says that all Lyapunov exponents of the cocycle $\mathcal{A}$ are equal to $\chi$ and, in particular, the cocycle is completely regular.

\smallskip

Now consider the second case, where there exists an ergodic Borel probability measure $\mu$ such that
$\tilde\chi_1(\mathcal{A},\mu) > \tilde\chi_d(\mathcal{A},\mu)$.
Then, for $\mu$-almost every point $y$,
\[
\lim_{n \to \infty}\frac{1}{n} \log \frac{\sigma_k(\mathcal{A}(y, n))}{\sigma_{k+1}(\mathcal{A}(y, n))}  = \tilde\chi_k(\mu) - \tilde\chi_{k+1}(\mu) > 0 \, .
\]
By minimality of $f$, the measure $\mu$ has full support, so the formula above holds for a dense set of points $y$.
On the other hand, by Theorem~\ref{t.zero_liminf}, there is another dense set where the formula fails, and actually the liminf of the sequence is $0$.
Now we can apply Theorem~\ref{thm-seq-phin} 
and conclude that the set of points $z$ for which 
\[
\lim_{n \to \infty}\frac{1}{n} \log \frac{\sigma_k(\mathcal{A}(z, n))}{\sigma_{k+1}(\mathcal{A}(z, n))}
\]
exists is meager in $X$.
In particular, the set $\mathcal{R}$ of LP-regular points is meager, as we wanted to show.
\end{proof}

At last, we can establish our general all-or-meager dichotomy for cocycles over minimal transformations:

\begin{proof}[Proof of Theorem~\ref{t.dichotomy_minimal}]
Let $\mathcal{A}$ be an arbitrary continuous cocycle over the minimal dynamics $f$.
Consider the finest dominated splitting of the cocycle (see \cite[Theorem~B.2]{BDV}; see also \cite{Sel,BlumLat}), that is, the invariant splitting $E_1 \oplus \cdots \oplus E_s$ which is dominated (i.e., satisfies inequalities \eqref{e.def_domination} for some $n_0$), and the number of bundles $s$ is maximal. The case $s=1$ corresponds to the situation where the cocycle admits no dominated splitting into non-trivial subbundles.  
The subbundles are continuous, by Proposition~\ref{p.DS_is_cont}. 
For each $i$, let $\mathcal{R}_i$ be the set of LP-regular points of the restriction of the cocycle to the $i$\textsuperscript{th} bundle $E_i$. By Lemma~\ref{l:minimal_no_dom}, 
each $\mathcal{R}_i$ is either a meager set or the whole $X$, and the latter alternative occurs only if the restricted cocycle $\mathcal{A}|E_i$ is completely regular.  
Note that the set of LP-regular points of $\mathcal{A}$ is $\mathcal{R} = \mathcal{R}_1 \cap \cdots \cap \mathcal{R}_k$. Therefore, $\mathcal{R}$ is either a meager set or the whole $X$. Furthermore, if $\mathcal{R} = X$, then for every $i$ we have  $\mathcal{R}_i = X$, which implies that each $\mathcal{A}|E_i$ is completely regular, and therefore $\mathcal{A}$ itself is completely regular. 
\end{proof}

%%%%%%%%%%%%%%%%%%%%%%%%%%%%%%%%%%%%%%%%%%%%%%%%%%%%%%%%%%%%%%%%%%%
\section{Further discussion of the notion of complete regularity} \label{s.independence}
%%%%%%%%%%%%%%%%%%%%%%%%%%%%%%%%%%%%%%%%%%%%%%%%%%%%%%%%%%%%%%%%%%%

In this section, we return to our definition of complete regularity. Recall that a continuous $d$-dimensional linear cocycle over a homeomorphism of a compact metric space is completely regular when
\begin{enumerate}
\item\label{i.again_CR1} every point is LP-regular, and
\item\label{i.again_CR2} the Lyapunov spectra of all (LP-regular) points are the same. 
\end{enumerate} 
In this section, we investigate the interdependency of these two conditions. We will construct examples showing that they are essentially independent -- in a sense to be made precise.

\subsection{All points being regular does not imply complete regularity}

By Theorem~\ref{t.dichotomy_minimal}, if the dynamics $f$ is minimal, then 
Condition~\eqref{i.again_CR1} in the definition of complete regularity does imply Condition~\eqref{i.again_CR2}, that is, if a cocycle  has the property that all points are LP-regular, then the Lyapunov exponents do not depend of the point. 

On the other hand, if $f$ is non-minimal, then Condition~\eqref{i.again_CR1} does not imply Condition~\eqref{i.again_CR2}. For example, let $f$ be the automorphism of the two-torus $(\R/\Z)^2$ induced by the matrix $\left( \begin{smallmatrix} 1 & 1 \\ 0 & 1 \end{smallmatrix} \right)$, that is, 
\begin{equation}\label{e.twist}
f(x,y) \coloneqq (x+y,y) \bmod \Z^2 \, .
\end{equation}
If we take a continuous cocycle of diagonal matrices, then every point will be LP-regular. But the Lyapunov exponents will not be constant, in general. This example was pointed out before in \cite[Example 4.1]{JohnZamp}.

Let us exhibit an example where $f$ is topologically transitive. Consider the group $G \coloneqq \mathrm{SL}(2,\R)$, the lattice $\Gamma \coloneqq \mathrm{SL}(2,\Z)$, and the quotient $N\coloneqq{_\Gamma}\backslash^G$, that is, the set of right cosets $\Gamma g$. Then $N$ is a Hausdorff but non-compact space with respect to the quotient topology (actually it is the unit tangent bundle of the modular surface). Let $X \coloneqq N \cup \{\infty\}$ be the one-point compactification of $N$. Define $f \colon X \to X$ by $f(\infty) \coloneqq \infty$ and 
\[
f(\Gamma g) \coloneqq  \Gamma g \left( \begin{smallmatrix} 1 & 1 \\ 0 & 1 \end{smallmatrix} \right) \, .
\]
The map $f$ is a homeomorphism and, like the previous example \eqref{e.twist}, is such the Birkhoff averages of every continuous function converge pointwise (see \cite[Theorem~5.1]{Dani}); furthermore, $f$ is topologically transitive. Arguing as before, Condition~\eqref{i.again_CR1} does not imply Condition~\eqref{i.again_CR2} for this map $f$.

The following result provides a similar example with a derivative cocycle of a diffeomorphism.

\begin{theorem}\label{thm-unipotent}
There exists a topologically transitive diffeomorphism $f: M\to M$ of a compact connected smooth Riemannian manifold $M$ of dimension $11$ such that every point $x\in M$ is LP-regular for the derivative cocycle and the set of Lyapunov vectors
$$
\big\{(\chi_1(x), \dots, \chi_{11}(x)) \in \R^{11} \ : \  x\in M\big \}
$$
has the power of the continuum.
\end{theorem}

For our construction we need the following objects:
\begin{enumerate}
	
\item The group $G \coloneqq \mathrm{SL}(3,\mathbb{R})$, which is a simple connected Lie group; we identify $\mathrm{SL}(2,\R)$ with the subgroup 
consisting of matrices of the form $\left( \begin{smallmatrix}  * & * & 0 \\ * & * & 0 \\ 0 & 0  & 1 \end{smallmatrix} \right)$.

\item A co-compact lattice $\Gamma\subset G$ such that $\Gamma\cap\mathrm{SL}(2,\R)$ is a co-compact lattice in $\mathrm{SL}(2,\R)$ (see \cite[Exerc.~4, {\S}18.6]{WMorris_groups}).

\item The quotient $N\coloneqq{_\Gamma}\backslash^G$, that is, the set of right cosets $\Gamma g$; this is a compact smooth manifold of dimension $8$; we denote by $m_N$ the normalized volume in $N$ which is induced on $N$ by the Haar measure on the group $G$. 

\item The one-parameter subgroup of unipotent matrices 
\begin{equation*}
u^t \coloneqq \begin{pmatrix} 1&t&0\\0&1&0\\0&0&1\end{pmatrix}
\end{equation*}
and the corresponding unipotent flow 
\begin{align*}
\varphi^t \colon N &\to N \\ 
\Gamma g &\mapsto \Gamma g u^t \, ,
\end{align*}
which preserves the volume measure $m_N$.

\item A $C^\infty$ strictly positive function $\tau:N\to\mathbb{R}$ whose choice will be specified later.

\item The Anosov flow $\{\psi^t\}$ which is a suspension flow with constant roof function $1$ over an Anosov automorphism of the two-dimensional torus $\mathbb{T}^2$ given by the matrix $A= \left( \begin{smallmatrix} 2&1\\1&1 \end{smallmatrix} \right)$; we denote by $S$ the $3$-dimensional suspension manifold; the flow $\{\psi^t\}$ preserves volume in $S$ which we denote by $m_S$.
\end{enumerate}

Denote $M \coloneqq N\times S$. Given $\alpha>0$, consider a skew-product map 
$f_\alpha: M\to M$ given by: 
$$
f_\alpha(x, y) \coloneqq (\varphi^1(x), \psi^{\tau(x)+\alpha}(y)).
$$
For every $\alpha$, the map $f_\alpha$ is a $C^\infty$ diffeomorphism preserving a smooth measure on $M$. We will show that with an appropriate choice of the parameter $\alpha$ and the function $\tau$, the map $f_\alpha$  has the desired properties. 

For every point $(x,y)\in M$ we have the following splitting of the tangent space 
\begin{equation}\label{e.splitting}
T_{(x,y)}M=T_{(x,y)} (N\times \{y\})\oplus E^\mathrm{s}_{f_\alpha}(x,y)\oplus E^\mathrm{u}_{f_\alpha}(x,y)\oplus E^0_{f_\alpha}(x,y),
\end{equation}
where 
\begin{equation}\label{e.splitting1}
E^\mathrm{s}_{f_\alpha}(x,y)=E^\mathrm{s}_{\psi^1}(y) \text{ and } E^\mathrm{u}_{f_\alpha}(x,y)=E^\mathrm{u}_{\psi^1}(y)
\end{equation}
are the one-dimensional stable and, respectively, the one-dimensional unstable subspaces  and 
\begin{equation}\label{e.splitting2}
E^0_{f_\alpha}(x,y)=\mathrm{Span}\left\{\tfrac{\partial}{\partial t}\right\}=E^0_{\psi^t}(y).
\end{equation} 

By Ratner's equidistribution theorem \cite[Theorem~B]{Ratner}, for every $x \in N$, the future and past orbits $\{\varphi^{\pm n} x\}_{n \ge 0}$ are equidistributed with respect to a probability measure $m_x$ on~$N$, that is, for every continuous function $h \colon N \to \R$, we have 
\begin{equation}\label{e.discrete_Ratner}
\lim_{n \to \infty} \frac{1}{n} \sum_{j=0}^{n-1} h(\varphi^{\pm j} x) = \int_N h \, dm_x \, .
\end{equation}
Furthermore (see \cite[p.~255]{Ratner}), the measures $m_x$ are ergodic, and thus $\varphi$ is \emph{pointwise ergodic} in the terminology of \cite{DownWeiss}.

The desired properties of the map $f_\alpha$ follow immediately from the following three lemmas.

\begin{lemma}\label{lem-continuum}
There is a smooth strictly positive function $\tau: N\to\mathbb{R}$ such that the set 
$$
\Bigl\{\int_{N}\tau\,dm_x: x \in N\Bigr\}
$$
has the power of continuum.
\end{lemma}

\begin{lemma}\label{lem-transitivity}
For every smooth strictly positive  function $\tau:N\to \mathbb{R}$ there exists a countable set $A \subset \R_+$ such that for every $\alpha \in \R_+ \setminus A$, the map $f_\alpha$ is topologically transitive.
\end{lemma}

\begin{lemma}\label{lem-formulas}
For any choice of a smooth function $\tau$ and any $\alpha\in(0,1)$ the splitting \eqref{e.splitting} is continuous, invariant under $df_\alpha$, and for every $(x,y)\in M$,
\begin{align}
\label{e.lyap-exp-1}
\lim_{n\to\pm\infty}&\frac1{n}\log\|Df_\alpha^n (v)\|=0 \quad \text{for every } v \in T_{(x,y)}(N\times\{y\}) \, ,
\\
\label{e.lyap-exp-2}
\lim_{n\to\pm\infty}&\frac1{n}\log\|Df_\alpha^n|E^\mathrm{s}_{f_\alpha}(x,y)\|
=-(\log\lambda)\int_{N}(\tau+\alpha)\,dm_x \, ,
\\
\label{e.lyap-exp-3}
\lim_{n\to\pm\infty}&\frac1{n}\log\|Df_\alpha^n|E^\mathrm{u}_{f_\alpha}(x,y)\|
=\phantom{-}(\log\lambda)\int_{N}(\tau+\alpha)\,dm_x \, ,
\\
\label{e.lyap-exp-4}
\lim_{n\to\pm\infty}&\frac1{n}\log\|Df_\alpha^n|E^0_{f_\alpha}(x,y)\|=0 \, ,
\end{align}
where $\lambda>1$ is the top eigenvalue of the matrix $A$. In particular, every point 
$(x,y)\in M$ is LP-regular.
\end{lemma}

We proceed with proofs of these three lemmas.

\begin{proof}[Proof of Lemma \ref{lem-continuum}]
Let $\tilde{G}$ be the subgroup of $G$ consisting of the matrices of the form $\left( \begin{smallmatrix}  * & * & 0 \\ * & * & 0 \\ 0 & 0  & 1 \end{smallmatrix} \right)$, that is, the standard copy of $\mathrm{SL}(2,\R)$ in $\mathrm{SL}(3,\R)$. Let $\pi \colon G \to N$ be the quotient map $\pi (g) \coloneqq \Gamma g$ and let $N_0 \coloneqq \pi(\tilde G)$. Since the unipotent subgroup $(u^t)$ is contained in $\tilde G$, the set $N_0$ is invariant under the unipotent flow $\{\varphi^t\}$. Recall that our lattice $\Gamma < G$ was chosen so that 
$\tilde \Gamma \coloneqq \Gamma \cap \tilde G$ is a co-compact lattice in $\tilde G$. Therefore, the quotient  $\tilde N\coloneqq{_{\tilde\Gamma}}\backslash^{\tilde G}$ is a compact $3$-manifold which is actually the unit tangent bundle of a compact hyperbolic surface. We define a unipotent flow on $\tilde N$ by the formula 
$\tilde{\varphi}^t (\tilde \Gamma \tilde g) \coloneqq \tilde \Gamma \tilde g u^t$; this is actually the horocycle flow on the unit tangent bundle. Define a map $\tilde{N} \to N_0$ by 
$\tilde \Gamma \tilde{g} \mapsto \Gamma \tilde{g}$. This map is a homeomorphism, and it conjugates the flows $\tilde{\varphi}^t$ and $\varphi^t|_{N_0}$. By a theorem of Furstenberg, the horocycle flow $\tilde{\varphi}^t$ is strictly ergodic (i.e., minimal and uniquely ergodic); furthermore, the time-$1$ map $\tilde{\varphi}^1$ is also strictly ergodic: see \cite{EP}. We conclude that the submanifold $N_0$ of $N$ supports a unique $\varphi^1$-invariant measure, which we call $m_{N_0}$. 

Next, consider the one-parameter subgroup of $G$ formed by the matrices of the form
\[
a^s \coloneqq \begin{pmatrix} e^s & 0 & 0 \\ 0 & e^s & 0 \\ 0 & 0 & e^{-2s} \end{pmatrix} \, , \quad s \in \R \, .
\]
Note that $a^s \tilde{g} = \tilde{g} a^s$ for every $\tilde{g} \in \tilde{G}$, and in particular, $a^s u^t = u^t a^s$.
As a consequence, the flow 
\begin{align*}
\zeta^s \colon N &\to N \\ 
\Gamma g &\mapsto \Gamma g a^s \, ;
\end{align*}
commutes with the unipotent flow, that is,
\[
\zeta^s \circ \varphi^t = \varphi^t \circ \zeta^s \, .
\]
Therefore, the submanifold $N_s \coloneqq \zeta^s(N_0)$ supports a unique $\varphi^1$-invariant measure, namely $m_{N_s} \coloneqq (\zeta^s)_*(m_{N_0})$. For every point $x$ in $N_s$, the measure $m_x$ defined by \eqref{e.discrete_Ratner} coincides with $m_{N_s}$.

We have constructed a continuous curve $s \mapsto m_{N_s}$ in the space of $\varphi^1$-invariant Borel probability measures on $N$ (endowed with the weak topology). We will prove that this curve is \emph{non-constant}. Once we do that, the lemma will immediately follow, as it will be sufficient to take any $C^\infty$ strictly positive function $\tau$ for which the function $s \mapsto \int_{N}\tau\,dm_{N_s}$ is non-constant. 

Note that $a^s \not \in \tilde G$ for all nonzero $s$. Since the projection $\pi \colon G \to N$ is a covering map and $\pi(\tilde{G})$ is the compact set $N_0$, it follows that 
$\pi(a^s) \not\in N_0$ for all sufficiently small $s \neq 0$. But $\pi(a_s) \in N_s$, and so $N_s \neq N_0$ for all sufficiently small $s \neq 0$. Since $N_s$ is the support of the measure $m_{N_s}$, this proves that the curve $s \mapsto m_{N_s}$ is non-constant, as we wanted to show.
\end{proof}

\begin{proof}[Proof of Lemma \ref{lem-transitivity}]
Consider the following two dynamical systems on the manifold $M = N \times S$:
\begin{itemize} 
	\item the flow $\mathrm{id} \times \psi^t$ whose time $t$ map is $(x,y) \mapsto (x,\psi^t(y))$;
	\item the diffeomorphism $f_0 (x,y) = (\varphi^1(x),\psi^{\tau(x)} y)$.
\end{itemize}
Note that the flow and the diffeomorphism above commute. Therefore, we obtain an action of the group $\R \times \Z$ on $M$. Note that the action preserves the product measure $m_N \times m_S$. 

We claim that the product measure is ergodic with respect to the $\R \times \Z$-action. Indeed, if $h \colon M \to \R$ is an invariant bounded Borel measurable function, then it is also invariant under the flow $\mathrm{id} \times \psi^t$. Since the measure $m_S$ is ergodic for the flow $\psi^t$, there exists some bounded Borel measurable function $h_1 \colon N \to \R$ such that $h(x,y) = h_1(x)$ for $m_N \times m_S$-almost every $(x,y)$. Since $h \circ f_0 = h$, we have $h_1 \circ \varphi^1 = h_1$ $m_N$-almost everywhere. Ergodicity implies that $h_1$ coincides with a constant $m_N$-almost everywhere. Hence, $h$ is constant $m_N \times m_S$-almost everywhere, completing the proof of ergodicity.

Note that the element $(\alpha,1)$ of the $\R \times \Z$-action is the diffeomorphism $f_\alpha$.
Let $A$ be the set of $\alpha \in \R_+$ such that $(f_\alpha, m_N \times m_S)$ is \emph{not} ergodic. By a result of Pugh and Shub \cite[Theorem~2]{PS}, the set $A$ is countable. Since the measure $m_N \times m_S$ is positive on open sets, the map $f_\alpha$ is topologically transitive whenever it is ergodic.
\end{proof}

\begin{proof}[Proof of Lemma \ref{lem-formulas}]
Equality \eqref{e.lyap-exp-1} follows from the definition of the map $f_\alpha$, and the following two facts: (a) all the Lyapunov exponents of $\varphi^t$ are zero; and (b) the Lyapunov exponent of the Anosov flow in its flow direction is zero. Equality \eqref{e.lyap-exp-4} follows from \eqref{e.splitting2} and the fact that the function $\tau$ is bounded. To prove equality \eqref{e.lyap-exp-2}, observe that for every $n>0$,
\[
f^n_\alpha(x,y)=(\varphi^n(x), \psi^{n\alpha+\tau_n(x)}(y)),
\]
where $\tau_n(x)=\sum_{k=0}^{n-1}\tau(\varphi^k(x))$. 
By property \eqref{e.discrete_Ratner}, for every $x \in N$
we have  
\[
\tau_n(x)=n\int_{N}\tau \,dm_x + o(n).
\]
The desired equality now follows from the fact that every $v\in E^\mathrm{s}(y)$ contracts under the Anosov flow by a constant factor $\lambda^t$. The proof of equality \eqref{e.lyap-exp-3} is similar. 
\end{proof}
We proved the lemmas and, as explained before, this completes the proof of Theorem~\ref{thm-unipotent}.

\subsection{Agreement of Lyapunov spectra does not imply complete regularity}

In this subsection we discuss the problem of whether Condition \eqref{i.again_CR2} alone is sufficient for complete regularity. Strictly speaking, as stated this condition subsumes the first and so the problem seems hollow. However, we may ask the following meaningful question: Suppose every two points that are LP-regular have the same Lyapunov spectrum. Does it follow that every point is LP-regular?

The answer is negative, as Example~\ref{ex.tangency} illustrates. In that example, most points are wandering; in fact, there is no invariant Borel probability measure whose support is $X$. 
The next result provides a more interesting example:

\begin{theorem}\label{t.2doesnotimply1}
There exists a homeomorphism $f \colon X \to X$ of a compact space $X$ admitting a fully-supported invariant Borel probability measure
and there exists a $2$-dimensional continuous linear cocycle $\mathcal{A}$ over $f$ with the following properties:
\begin{itemize}
\item all LP-regular points $x \in X$ have Lyapunov exponents
$\chi_1(x) = c$ and $\chi_2(x) = -c$, where $c$ is a positive constant;
\item there exist points in $X$ that are not LP-regular.
\end{itemize}
\end{theorem}

In fact, the map $f$ in our example is uniquely ergodic, and so, by Theorem~\ref{t.dichotomy_minimal}, points that are not LP-regular for the cocycle $\mathcal{A}$ form a residual subset of $X$.

We do not know if there exist examples of cocycles over a hyperbolic base dynamics satisfying the conclusions of Theorem~\ref{t.2doesnotimply1}. If that is the case, the matrix maps cannot be H\"older continuous, by Theorem~\ref{t.amenable_reduction}.

Another question is the following: Assuming Condition \eqref{i.again_CR1} in the definition of complete regularity, can Condition \eqref{i.again_CR2} be weakened to the following condition~(2')?
\begin{itemize}
	\item[(2')] all ergodic measures have the same Lyapunov spectrum.
\end{itemize}

We now turn to the proof of Theorem~\ref{t.2doesnotimply1}. 

If  $f \colon X \to X$ is a homeomorphism of a compact metric space $X$, we say that $f$ is a \emph{Veech-like map} if it is strictly ergodic (i.e.\ minimal and uniquely ergodic) and $f^2$ is minimal but \emph{not} uniquely ergodic. Veech-like maps were first constructed by Veech \cite{Veech69}. 

Walters \cite{Walters} has shown that given a Veech-like map $f$, one can construct a non-uniformly hyperbolic cocycle with base $f$. We will construct a specific Walters cocycle with some additional properties which will be  instrumental in our proof of  Theorem~\ref{t.2doesnotimply1}. For this cocycle we need the base Veech-like map constructed in \cite[Sec.~4]{mono} (the idea of this construction is due to Scott Schmieding).

\begin{proof}[Proof of Theorem~\ref{t.2doesnotimply1}]
Let $e_1, e_2, \dots$ be words in the alphabet $\mathsf{A} \coloneqq  \{ \mathord{\uparrow}, \mathord{\downarrow} , 0 \}$, defined in a recursive manner as follows: 
\begin{align}
\notag
e_{1}	&\coloneqq \mathord{\uparrow} \mathord{\downarrow}  \, ,  
\\ 
\label{e.next_word}
e_{k+1} &\coloneqq 
e_k^{k^2+1} \, 0 \, e_k \, 
\overline{e_k}^{k^2+1} \,  0 \, \overline{e_k} \, ,
\end{align}
where $\overline{e_k}$ denotes the word obtained from $e_k$ by switching the arrows.
Let $X$ be the subset of $\mathsf{A}^\Z$ formed by those bi-infinite sequences $\omega=(\omega_i)_{i\in\Z}$ 
such that for every $n > 0$, the word $\omega_{-n}\omega_{-n+1}\cdots\omega_n$ occurs as a subword (i.e., a factor) of some $e_k$.
According to \cite[Theorem~4.1]{mono}, 
$X$ is a closed, shift-invariant subset of $\mathsf{A}^\Z$, and if $f \colon X \to X$ is the restricted shift, then $f$ is a Veech-like map. 
Define a continuous matrix valued map 
\[
A(\omega) \coloneqq \begin{pmatrix} 0 & e^{\varphi(\omega)} \\ e^{-\varphi(\omega)} & 0 \end{pmatrix} 
\quad \text{where} \quad
\varphi(\omega) \coloneqq
\begin{cases}
\phantom{-}1	&\text{if } \omega_0 = \mathord{\uparrow} \, , \\ 
-1 				&\text{if } \omega_0 = \mathord{\downarrow} \, , \\
\phantom{-}0	&\text{if } \omega_0 = 0 \, .
\end{cases}
\]
Then \cite[Theorem~4.11]{mono} states that the corresponding cocycle $\mathcal{A}$ has nonzero Lyapunov exponents (with respect to the unique $f$-invariant measure), and it is not uniformly hyperbolic.
Note that for any $\omega \in X$ and $n>0$, 
\begin{align*}
\frac{1}{n}\log \left\|\mathcal{A}(\omega,n)  \right\| &= 
\frac{1}{n} \left| \sum_{j=0}^{n-1}(-1)^j \varphi(f^j(\omega)) \right| \, , \\ 
\frac{1}{n} \log \left\|\mathcal{A}(\omega,-n) \right\| &= 
\frac{1}{n} \left| \sum_{j=-n}^{-1}(-1)^j \varphi(f^j(\omega)) \right| \, .
\end{align*}
If $\omega$ is a LP-regular point, then the two sequences above must converge to the same limit as $n \to \infty$.
However, according to \cite[Lemma~4.12]{mono}, this common limit, whenever it exists, is a number $c>0$ independent of $\omega$.
Therefore all LP-regular points of our Walters cocycle have the same Lyapunov exponents $c$ and $-c$. 

At the same time, since $f$ is uniquely ergodic and $\mathcal{A}$ is not uniformly hyperbolic, we must be in case~\eqref{i:NUH} of Proposition~\ref{p.trichotomy}. Consequently, $\mathcal{A}$ cannot be completely regular (even though, as we saw above, all LP-regular points have the same spectra).
 \end{proof}

%%%%%%%%%%%%%%%%%%%%%%%%%%%%%%%%%%%%%%%%%%%%%%%%%%%%%%%%%%%%%%%%%

\bigskip

\noindent \textbf{Acknowledgements.} %: acknowledgements
We express our thanks to Boris Kalinin and Federico Rodriguez Hertz for helpful discussions and to Mat\'{\i}as Ures for pointing out the missing requirement of uniform convergence in the definition of regularity. 
We thank the referee for their careful revision and suggestions. 
Y.P.\ was partially supported by NSF grant DMS 2153053.
O.S.\ was partially supported by ISF grant 264/22 and by Minerva grant ``Chaotic non-Anosov dynamical systems.''

%%%%%%%%%%%%%%%%%%%%%%%%%%%%%%%%%%%%%%%%%%%%%%%%%%%%%%%%%%%%%%%%%

\end{document}